\documentclass[11pt]{article}

\usepackage[margin=2.4cm]{geometry}
\usepackage{amsmath,amssymb}

\usepackage{amsthm}

\numberwithin{equation}{section}

\usepackage{enumitem,color}
\usepackage{hyperref}

\usepackage{multirow,booktabs}
\usepackage{graphicx}
\usepackage{subcaption}

\usepackage{graphicx}%
\usepackage{multirow}%
\usepackage{amsmath,amssymb,amsfonts}%
\usepackage{amsthm}%
\usepackage{mathrsfs}%
\usepackage[title]{appendix}%
\usepackage{xcolor}%
\usepackage{textcomp}%
\usepackage{manyfoot}%
\usepackage{booktabs}%
\usepackage{algorithm}%
\usepackage{algorithmicx}%
\usepackage{algpseudocode}%
\usepackage{listings}%
\usepackage{bm}

\newtheorem{theorem}{Theorem}[section]

\newtheorem{lemma}[theorem]{Lemma}

\theoremstyle{definition}

\newtheorem{remark}[theorem]{Remark}

\newcommand{\ds}{\displaystyle}
\newcommand{\f}{\frac}

\newcommand{\om}{\Omega}
\newcommand{\p}{\partial}

\title{Optimal control of variable-exponent subdiffusion}

\author{
Yiqun Li\thanks{School of Mathematics and Statistics, Wuhan University, Wuhan 430072, China. Email: YiqunLi24@outlook.com.}
\and Mengmeng Liu\thanks{School of Mathematics, Shandong University, Jinan 250100, China. Email: liumengmeng423@163.com}
\and
Wenlin Qiu\thanks{Corresponding author. School of Mathematics, Shandong University, Jinan 250100, China. Email: wlqiu@sdu.edu.cn.}
}

\date{}

\begin{document}

\maketitle

\begin{abstract}
This work investigates the optimal control of the variable-exponent subdiffusion, which extends the work [Gunzburger and Wang, {\it SIAM J. Control Optim.} 2019] to the variable-exponent case to account for the multiscale and crossover diffusion behavior. To resolve the difficulties caused by the leading variable-exponent operator, we adopt the convolution method to reformulate the model into an equivalent but more tractable form, and then prove the well-posedness and weighted regularity of the optimal control. As the convolution kernels in reformulated models are indefinite-sign, non-positive-definite, and non-monotonic, we adopt the discrete convolution kernel approach in numerical analysis to show the $O(\tau(1+|\ln\tau|)+h^2)$ accuracy of the schemes for state and adjoint equations. Numerical experiments are performed to substantiate the theoretical findings.

\end{abstract}

\section{Introduction}
\subsection{Model problem}
Optimal control  has widespread applications in various fields
 and thus has been  extensively investigated \cite{Delia,Gun,Herzog,ItoKun,KunVex}.
In particular, optimal control of subdiffusion \cite{AntOtaSal,Ant2,JinLiZouSINUM,Mop,WanLi,ZhouGong}
 has attracted increasing attentions since the subdiffusion equation adequately models challenging phenomena such as anomalously diffusive transport through the highly heterogeneous porous media \cite{DieFor,Div,FuZhu,Kop,MetKla}. In complex processes or time-dependent surroundings, there often exists multiscale diffusion or crossover of diffusive regimes, which could be accommodated by introducing the variable exponent in the subdiffusion equation \cite{DuZho,FanHu,GarGiu,Sib,ZheLiQiu}. Thus, we aim to study the corresponding optimal control problem for practical applications.

 Let $u$ be the concentration of the mixed fluid of interest and $c$ be the control variable from the following admissible set
$$U_{ad} :=  \big \{c\in  L^2(0,T;L^2(\om)): \int_\Omega c(\bm x,t)d\bm x\ge 0,~~\forall t \in[0,T]\}.$$
We thus consider the following optimal control model
\begin{equation}\label{ObjFun}
\min_{c \in U_{ad},u} J(u,c)=\frac{1}{2}\|u-u_d\|^2_{L^2(0, T; L^2(\Omega))}+\frac{\kappa}{2}\|c\|^2_{L^2(0, T; L^2(\Omega))},
\end{equation}
where  $u_d$ and $\kappa>0$  represent the target state and the penalty parameter, respectively. The minimization problem \eqref{ObjFun} is
governed by the variable-exponent subdiffusion equation \cite{SunCha,ZheLiQiu}
\begin{equation}\label{VtFDEs}
\begin{array}{c}
\p_t u(\bm x,t)- {}^R\p_t^{1-\alpha(t)} \Delta u (\bm x,t)= q(\bm x,t) + c(\bm x,t)  ,~~(\bm x,t) \in \Omega\times(0,T]; \\ [0.05in]
\ds u(\bm x,0)=0,~\bm x\in \Omega; \quad u(\bm x,t) = 0,~(\bm x,t) \in \p \Omega\times[0,T].
\end{array}\end{equation}
Here
$\Omega \subset \mathbb{R}^d$  is a simply-connected bounded domain with the piecewise smooth boundary $\p \om$ with convex corners, $\bm x := (x_1,\cdots,x_d)$ with $1 \le d \le 3$ denotes the spatial variable, and $q$ refers to the source term.  ${}^R\p_t^{1-\alpha(t)}$ with $0< \alpha(t)<1$ is the variable exponent Riemann--Liouville fractional
differential operator defined as follows ($*$ represents the symbol of convolution)
\cite{LorHar}
\begin{align}\label{var_opr}
&\ds {}^R\p_t^{1-\alpha(t)}   \psi = \p_t I_t^{\alpha(t)} \psi, \quad  I_t^{\alpha(t)} \psi: =k(t)*  \psi(\bm x,t), \quad k(t) : =\frac{t^{ \alpha(t)-1}}{\Gamma(\alpha(t))}. \vspace{-0.2in}
\end{align}

\subsection{Main contributions}
In contrast to the constant-exponent case \cite{Mus}, there are much less investigations on  the variable-exponent subdiffusion and its optimal control. In \cite{Li,ZheWanSICON}, some related variable-exponent time-fractional optimal control problems have been considered where both the variable-exponent time-fractional derivative and the first-order time derivative exist in the model. For this case, the variable-exponent fractional term, though does not have favorable properties as its constant-exponent analogue, is a low-order term such that the perturbation analysis could be applied. In \cite{ZheLiQiu}, the model (\ref{VtFDEs}) with $\alpha(0)=1$ has been considered. For this special case, the properties of model (\ref{VtFDEs}) is close to  its integer-order analogue since the initial singularity of the solutions has been locally eliminated by $\alpha(0)=1$. For the general case $0<\alpha(t)<1$,
 \cite{Zheng} proposed a convolution method to reformulate a Caputo-type variable-exponent subdiffusion
   $ k*\p_t u -\Delta u = f,$
   and then prove its well-posedness by the Laplace transform method under the condition $f \in W^{1,1}(0, T; L^2(\Omega))$.

In this work, we adopt the convolution method proposed in \cite{Zheng} to transform the state and adjoint equations into the equivalent but more tractable formulations, and then apply more subtle estimates to reduce the smoothness condition of the forcing term from $W^{1,1}(0, T; L^2(\Omega))$ in \cite{Zheng} to $ L^2(0, T; L^2(\Omega))$ in the well-posedness proof.
We also derive weighted regularity estimates of the solutions under weaker assumptions on the data than \cite{Zheng}, where the exponent of the weight function is carefully selected according to the feature of optimal control.

For numerical analysis, the widely-used approach such as the induction argument \cite{Stynes} could not be directly applied since the convolution kernels in reformulated models lack desirable properties such as the positive-definiteness and monotonicity.
 In a recent work \cite{QiuZhe}, the discrete convolution kernel method under the graded mesh \cite{LiLiZha} is introduced to prove the first-order temporal accuracy of the numerical scheme to a subdiffusion model with non-positive memory. In the current work, we follow the idea of \cite{QiuZhe} to develop numerical schemes for state and adjoint equations but relax the mesh condition based on the enhanced solution regularity, that is, we prove the almost first-order accuracy (that is, $\tau(1+|\ln\tau|)$-order accuracy) of numerical schemes for the reformulated models under the uniform mesh.

 The rest of the work is organized as follows: In \S \ref{S:Pre}, we derive the first-order optimality condition of the optimal control. In \S \ref{Sect:Aux}, we prove several estimates for an auxiliary equation. In \S \ref{Sect:state}, we analyze the state and adjoint  equations as well as the optimal control.
The numerical scheme for the state and adjoint  equations are developed and  analyzed in \S \ref{S:Num}. Numerical experiments are performed in the last section to simulate the optimal control. Some concluding remarks are given in the last section.

\section{Reformulation of optimal control} \label{S:Pre}

\subsection{Notations and preliminaries}
Let $L^p(\om)$ with $1 \le p \le \infty$ be the Banach space of $p$th power Lebesgue integrable functions on $\om$. For a positive integer $m$,
let  $ W^{m, p}(\Omega)$ be the Sobolev space of $L^p$ functions with $m$th weakly derivatives in $L^p(\om)$ (similarly defined with $\om$ replaced by an interval $\mathcal I$). Let  $H^m(\Omega) := W^{m,2}(\Omega)$ and $H^m_0(\Omega)$ be its subspace with the zero boundary condition up to order $m-1$. For a non-integer $s\geq 0$, $H^s(\Omega)$ is defined via interpolation \cite{AdaFou}. Let $\{\lambda_i,\phi_i\}_{i=1}^\infty$ be eigenpairs of the problem $-\Delta \phi_i = \lambda_i \phi_i$ with the zero boundary condition where $\{\phi_i\}_{i=1}^\infty$ form an orthonormal basis for $L^2(\Omega)$. We introduce the Sobolev space $\check{H}^s(\Omega)$ for $s\geq 0$ by
$ \check{H}^{s}(\Omega) :=  \{ v \in L^2(\Omega): \| v \|_{\check{H}^s}^2 : = \sum_{i=1}^{\infty} \lambda_i^{s} (v,\phi_i)^2 < \infty\}$,
which is a subspace of $H^s(\Omega)$ satisfying $\check{H}^0(\Omega) = L^2(\Omega)$ and $\check{H}^2(\Omega) = H^2(\Omega) \cap H^1_0(\Omega)$ \cite{Jinbook}.
For a Banach space $\mathcal Y$, let $W^{m, p}(0,T; \mathcal Y)$ be the space of functions in $W^{m, p}(0,T)$ with respect to $\|\cdot\|_{\mathcal Y}$.  All spaces are equipped with standard norms \cite{AdaFou,Eva}.
We use $Q$, $Q_i$, and $Q_*$ to denote generic positive constants in which $Q$ may assume different values at different occurrences. We set $\|\cdot\|:=\|\cdot\|_{L^2(\Omega)}$ and $L^p(\mathcal Y)$ for $L^p(0,T;\mathcal Y)$ for brevity, and drop the notation $\om$ in the spaces and norms if no confusion occurs.

Let $E_{p, \bar{p}}(z)$ be the Mittag-Leffler function
$E_{p, \bar{p}}(z):=\sum_{k=0}^{\infty} \frac{z^k}{\Gamma(p k+\bar{p})}$
for $z \in \mathbb{R}$, $p \in \mathbb{R}^{+}$, and $\bar{p} \in \mathbb{R}$.
We refer to some lemmas from \cite{Jinbook} for future use.

\begin{lemma}\label{Lem1}
Let $0<p<2$ and $\bar{p} \in \mathbb{R}$, and   $\pi p / 2<\mu<\min \{\pi, \pi p\}$. Then there exists a constant $C_1=C_1(p, \bar{p} ,\mu)>0$ such that
$\big|E_{p, \bar{p}}(z)\big| \le  C_1 (1+|z|)^{-1}$  for $\mu \leqslant|\arg (z)| \leqslant \pi$.
\end{lemma}

\begin{lemma}\label{Lem2}
For $p\in (0,1)$  and $\lambda \in \mathbb{R}$, we get $\frac{\mathrm{d} }{\mathrm{~d} t } E_{p, 1}\left(-\lambda t^p\right)=-\lambda t^{p-1} E_{p, p}\left(-\lambda t^p\right)$, $\frac{\mathrm{d}^k}{\mathrm{~d} t^k}\left(t^{\bar{p}-1} E_{p, \bar{p}}\left(-\lambda t^p\right)\right)=t^{\bar{p}-k-1} E_{p, \bar{p}-k}\left(-\lambda t^p\right)$ for $k \in \mathbb N^+$ and $\bar p>0$, as well as $I_t^\epsilon\left(t^{\bar{p}-1} E_{p, \bar{p}}\left(-\lambda t^p\right)\right)=t^{\bar{p}+\epsilon-1} E_{p, \bar{p}+\epsilon}\left(-\lambda t^p\right)$ for $\epsilon > 0$.
\end{lemma}

Throughout this work, we assume  $0 <\alpha( t)   <1$ and $|\alpha^{\prime}(t)|   \le Q_*$ on $[0, T]$.

\subsection{First-order optimality condition} We formally derive the first-order optimality condition for the optimal control \eqref{ObjFun}--\eqref{VtFDEs} in the following theorem. The well-posedness of state and adjoint equations will be proved later such that the derivations of the first-order optimality condition can be justified.
\begin{theorem}\label{thm:OptCond}
Suppose $q,c\in L^2(L^2)$,
there exists an adjoint state $z$ such that $(u,c,z)$ satisfies the state equation (\ref{VtFDEs}) and the adjoint equation
\begin{equation}\label{AdjEq}\begin{array}{c}
-\p_t z-  {}^c\hat\p_t^{1-\alpha(t)} \Delta z   = u(\bm x,t;q)-u_d(\bm x,t), ~(\bm x,t) \in \Omega\times[0,T); \\ [0.05in]
\ds z(\bm x,T)=0, ~ \bm x\in \Omega; \quad z(\bm x,t) = 0, ~(\bm x,t) \in \p \Omega\times[0,T]
\end{array}\end{equation}
with the variational inequality
\begin{equation}\label{VarIneq}
\int_0^T\int_\Omega(\kappa c+z)(v-c)d\bm x dt\geq 0, \qquad \forall  v \in U_{ad}.
\end{equation}
Here the backward variable-exponent Caputo differential operator ${}^c \hat\p_t^{1-\alpha(t)}$ is the adjoint operator of ${}^R\p_t^{1-\alpha(t)} $ defined as
\begin{equation}\label{VtInt}
 {}^c\hat\p_t^{1-\alpha(t)}\psi:= -{}_tI_T^{\alpha(t)} \p_t \psi, \quad {}_tI_T^{\alpha(t)} \psi(t): = \int_t^T \ds \frac{\psi( s)}{\Gamma(\alpha(s-t))(s-t)^{1-\alpha( s-t)}}ds.
\end{equation}

\end{theorem}

\begin{proof}
The proof of \eqref{AdjEq}--\eqref{VarIneq} could be carried out by   following a similar procedure of \cite[Theorem 2.1]{ZheWanSICON}  and   the adjoint properties of ${}^c\hat \p_t^{1-\alpha(  t)}\Delta$ and ${}^R  \p_t^{1-\alpha( t)}\Delta$, that is,
\begin{equation*}\label{Jv:e4}\begin{array}{ll}
& \hspace{-0.1in} \ds \ds \int_{\Omega}\int_0^T   u \,{}^c\hat\p_t^{1-\alpha( t)} \Delta z   dt = -\int_\Omega\int_0^T u(\bm x, t)\bigg[\int_t^T \ds \frac{\p_s \Delta z(\bm x, s)}{\Gamma(\alpha(s-t))(s-t)^{1-\alpha( s-t)}}ds \bigg]dt d \bm x\\[0.1in]
&\hspace{-0.1in} \ds\!=\! -\int_{\Omega}\int_0^T \p_s z(\bm x, s) \bigg[\int_0^s \frac{\Delta u(\bm x, t)}{\Gamma(\alpha(s-t))(s-t)^{1-\alpha( s-t)}} dt\bigg] dsd \bm x\\[0.1in]
&\hspace{-0.1in} \ds \!=  \!\int_{\Omega}\int_0^T  z(\bm x, s)\bigg[ \p_s  \int_0^s \frac{\Delta u(\bm x, t)}{\Gamma(\alpha(s-t))(s-t)^{1-\alpha( s-t)}} dt\bigg] dsd \bm x\!=\!\int_{\Omega}\int_0^T  z\, {}^R\p_s^{1-\alpha(s)} \Delta u dsd \bm x,
\end{array}\end{equation*}
and the detailed proof is   omitted due to similarity.
\end{proof}

\begin{remark}
 The variational inequality \eqref{VarIneq} could be interpreted as follows \cite{ZheWanSICON}
\begin{equation}\label{VarC}
\kappa c(\bm x,t) = \max \bigg\{0,\frac{1}{|\Omega|}\int_\Omega z(\bm x,t) d\bm x \bigg \} - z(\bm x,t).
\end{equation}
\end{remark}

\subsection{Equivalent reformulation}\label{Sect:Equiv}
For the sake of analysis, we introduce a generalized identity function, which is the convolution of the variable-exponent kernel \eqref{var_opr} and the Abel kernel $ \beta_{1-\alpha_0}(t) : = \f{t^{-\alpha_0}}{\Gamma{(1-\alpha_0)}}$ with $\alpha_0 = \alpha(0)$
\begin{align}
\hspace{-0.175in} g (t) \!: =\! \left(\beta_{1-\alpha_0} * k\right)(t) & \!=\!\int_0^t \frac{(t-s)^{-\alpha_0}}{\Gamma(1-\alpha_0)} \frac{s^{\alpha(s)-1}}{\Gamma(\alpha(s))} d s \!=\!\int_0^1 \frac{(t-t z)^{-\alpha_0}}{\Gamma(1-\alpha_0)} \frac{(t z)^{\alpha(t z)-1}}{\Gamma(\alpha(t z))} t d z  \nonumber \\
& \!=\!\int_0^1 \frac{(t z)^{\alpha(t z)-\alpha_0}}{\Gamma(1-\alpha_0) \Gamma(\alpha(t z))}(1-z)^{-\alpha_0} z^{\alpha_0-1} d z,  \vspace{-0.1in}\label{g}
\end{align}
where we apply the variable substitution $z = s/t$ on the right-hand
side of the second equality.  It is shown in \cite{Zheng} that
\begin{equation}\label{g:est}
\ds g(0)=1, \quad   |g(t)| \leq Q, \quad |g^{\prime}(t) | \leq Q(|\ln t|+1), ~~\mbox{on}~ (0, T].
\end{equation}

Now we intend to propose an equivalent formulation of (\ref{VtFDEs}).
\begin{theorem}\label{equi01}
Suppose $u  \in H^1(L^2)\cap L^2(\check H^2)$. Then $u$ solves the state equation \eqref{VtFDEs} if and only if $u$ solves the following system
\begin{equation}\label{Model}\begin{array}{l}
\hspace{-0.1in}\ds ^c\p_t^{\alpha_0} u(\bm x,t)\!-\!  \Delta u (\bm x,t) \!-\! g^\prime * \Delta u (\bm x,t) \!=\! I_t^{1-\alpha_0} (q+c)(\bm x,t),~~(\bm x,t) \in \Omega\times(0,T]; \\ [0.05in]
\ds u(\bm x,0)=0,~\bm x\in \Omega; \quad u(\bm x,t) = 0,~(\bm x,t) \in \p \Omega\times[0,T].
\end{array}\end{equation}
\end{theorem}

\begin{proof}
We first prove the problem \eqref{VtFDEs} could be reformulated to \eqref{Model}.
We calculate the convolution of \eqref{VtFDEs} and $\beta_{1-\alpha_0}$ as
$\beta_{1-\alpha_0} *\big[\p_t u -  \p_t (k * \Delta   u)- (q+c) \big]=0$.
Since $u  \in H^1(L^2)\cap L^2(\check H^2)$, we obtain $u \in C(H^1)$ \cite{Eva}. We incorporate this with \eqref{var_opr} and  H\"older's inequality to obtain   $(k*u)\big|_{t=0} = 0$, which further gives
$$(k * \Delta  u)\big|_{t=0} =  \Delta (k *   u)\big|_{t=0} =0.$$
We combine this with \eqref{g} to obtain
\begin{equation}\label{thm1:e2}
\beta_{1-\alpha_0} * \big(\p_t (k * \Delta   u)\big) = \p_t (\beta_{1-\alpha_0} *k * \Delta   u) =  \p_t (g * \Delta   u) = \Delta u + g^{\prime} * \Delta   u.
\end{equation}
We incorporate this with \eqref{var_opr}  to get \eqref{Model}.

Conversely, suppose $u$ is the solution to \eqref{Model}. By \eqref{thm1:e2}, we compute the convolution of  \eqref{Model} and $\beta_{\alpha_0}$, and then differentiate the resulting equation  to obtain \vspace{-0.1in}
\begin{equation*}\begin{array}{cl}
\hspace{-0.1in} \ds 0 \!=\! \p_t \big[(\beta_{\alpha_0}* \beta_{1-\alpha_0}) *\big(\p_t u \! -\!   \p_t (k * \Delta   u)\! -\!  (q+c) \big)\big] & \ds\! = \! \p_t \big[1 *\big(\p_t u -  \p_t (k * \Delta   u)- (q+c)\big)\big], \vspace{-0.1in}
\end{array}
\end{equation*}
which recovers the state equation \eqref{VtFDEs}, and thus completes the proof.
\end{proof}

To facilitate analysis the adjoint problem \eqref{AdjEq}, we apply the variable substitution $\bar t = T-t$  to obtain a forward-in-time analogue of \eqref{VtInt}
\begin{equation}\label{AdjCap}\begin{array}{ll}
\hspace{-0.2in}\ds {}^c\hat\p_t^{1-\alpha(t)}\Delta z &\ds \!=\!  \!- \!\int_t^T \ds \frac{\p_s \Delta z(\bm x, s)ds}{\Gamma(\alpha(s-t))(s-t)^{1-\alpha( s-t)}} \!= \!  \int_0^{\bar t} \ds \frac{\p_{y} \Delta \bar z(\bm x, y)d y}{\Gamma(\alpha(\bar t- y))(\bar t - y)^{1-\alpha(\bar t- y)}} \\[0.15in]
&\hspace{-0.125in}\ds \ds =I_{\bar t}^{\alpha(\bar t)} \p_{\bar t} \Delta \bar z =: {}^c\bar\p_{\bar t}^{1-\alpha(\bar t)} \Delta \bar z, \quad \bar z(\bm x, \bar t) : = z (\bm x, T-\bar t).
\end{array}\end{equation}
We invoke this to arrive at
a forward-in-time analogue of the adjoint problem \eqref{AdjEq}
\begin{equation}\label{AdjEq1}\begin{array}{c}
\p_{\bar t} \bar z-  {}^c\bar\p_{\bar t}^{1-\alpha(\bar t)} \Delta \bar z   = \bar u(\bm x, \bar t)- \bar {u}_d (\bm x,\bar t), ~(\bm x, \bar t) \in \Omega\times(0,T]; \\ [0.05in]
\ds \bar z(\bm x,0)=0, ~ \bm x\in \Omega; \quad \bar z(\bm x, \bar t) = 0, ~(\bm x, \bar t) \in \p \Omega\times[0,T].
\end{array}\end{equation}
Here  $\bar u$  and $ \bar {u}_d$ could be  defined in an analogous manner as $\bar z$ in \eqref{AdjCap}.

\begin{theorem}\label{Ref:Adj}
Suppose $\bar z  \in H^1 (L^2)\cap L^2(\check H^2)$. If $\bar z$ solves the  model \eqref{AdjEq1}, then $\bar z$ solves the following system
\begin{equation}\label{AdjEq2}\begin{array}{l}
\hspace{-0.175in}\ds ^c\p_{\bar t}^{\alpha_0} \bar z(\bm x,\bar t)\!-\!  \Delta \bar z (\bm x,\bar t) \!-\! g^\prime * \Delta \bar z (\bm x,\bar t) \!=\! I_{\bar t}^{1-\alpha_0} (\bar u -\bar u_d)(\bm x,\bar t),~(\bm x,\bar t) \in \Omega\times(0,T]; \\ [0.05in]
\ds~~~~~~~~~ \bar z(\bm x,0)= 0,~\bm x\in \Omega; \quad \bar z(\bm x, \bar  t) = 0,~(\bm x, \bar t) \in \p \Omega\times[0,T].
\end{array}\end{equation}
 Conversely, if $\bar z$ solves \eqref{AdjEq2}, then $\bar z$ solves the  model  \eqref{AdjEq1}.
\end{theorem}

\begin{proof}
The proof is similar to that of Theorem \ref{equi01} and is thus omitted.
\end{proof}
\section{Analysis of auxiliary equation}\label{Sect:Aux}

We  analyze the well-posedness and  regularity estimates of the solutions to an auxiliary differential equation
\begin{equation}\begin{array}{l}\label{ode}
^c\p_t^{\alpha_0}  v + \lambda v + \lambda g^\prime *  v = I_t^{1-\alpha_0} f,\quad t \in (0, T], \quad v(0)=0,
  \end{array}
 \end{equation}
where $ ^c\p_t^{\alpha_0} $ is the Caputo fractional  derivative operator defined as
$^c\p_t^{\alpha_0 } v : = \beta_{1-\alpha_0}*  v^\prime( t)$ \cite{Jinbook} and $f$ is a forcing term.

\begin{theorem}\label{thm:ode}
Suppose  $f \in L^2(0,T)$, then the auxiliary differential equation \eqref{ode} admits a unique solution in $H^1(0,T)$ such that
\begin{equation}\label{odestab}\begin{array}{cl}
&\hspace{-0.1in} \ds  \| v\|_{H^1(0,T)} \leq Q  \|f\|_{L^2(0, T)}.
\end{array}
\end{equation}
 Suppose $f \in H^1(0, T)$, we have
 \begin{equation}\label{thm2:est}\begin{array}{cl}
  \ds  \big\| t^{\vartheta} v^{\prime \prime} \big\|_{L^2(0,T)} \leq Q ( \|f\|_{H^1(0, T)} + \lambda^{1-\gamma} |f(0)|)
 \end{array}
 \end{equation}
with\vspace{-0.1in}
\begin{equation}\label{theta}
\gamma \in \Big(\f{3}{4}, 1\Big),~~\vartheta \in  \Big( \vartheta _*, \f{1}{2}\Big), \,\, \vartheta_* : = \max \Big\{\f{1}{2} -\alpha_0(1-\gamma), \alpha_0 -\f{1}{2}\Big\} \in \Big(\f{3}{10}, \f{1}{2}\Big).
\end{equation}
\end{theorem}
\begin{proof}
To prove the well-posedness, we define the space $\mathcal{X}:=\{g \in H^1(0, T): g(0)=0\}$ equipped with the equivalent norm  $\|g\|_{\mathcal{X}, \sigma}:=$ $ \|e^{-\sigma t} \p _t g \|_{L^2(0, T)}$ for some $\sigma \ge 0$.  For each $w \in \mathcal{X}$, let $v : = \mathcal{M} v$ be the solution of
\begin{equation*}\label{ode:e1}
 ^c\p_t^{\alpha_0}  v + \lambda v  = -\lambda g^\prime *  w +  I_t^{1-\alpha_0} f,
\end{equation*}
 which could be expressed as
$ \ds  v=(t^{\alpha_0-1} E_{\alpha_0, \alpha_0}(-\lambda  t^{\alpha_0})) *(-\lambda g^{\prime} * w +  I_t^{1-\alpha_0} f)$ \cite{Jinbook},
where the last term could be further reformulated  as
$\big[(t^{\alpha_0-1} E_{\alpha_0, \alpha_0}(-\lambda  t^{\alpha_0}))  *  \beta_{1-\alpha_0}\big] * f = E_{\alpha_0, 1}(-\lambda  t^{\alpha_0})* f$.
We incorporate this with Lemma \ref{Lem2} to differentiate the solution expression to obtain
\begin{equation}\begin{array}{l}\label{diff:sol}
 \hspace{-0.1in} \ds  v^{\prime}= -\lambda  (t^{\alpha_0-1} E_{\alpha_0, \alpha_0}(-\lambda  t^{\alpha_0})) *( g^{\prime} * w^{\prime})  + f -\lambda (t^{\alpha_0-1} E_{\alpha_0, \alpha_0}(-\lambda  t^{\alpha_0})) * f.
  \end{array}
\end{equation}
To bound  the first right-hand side term, we invoke the complete monotonicity property of $E_{\alpha_0, \alpha_0}(-\lambda  t^{\alpha_0})$ \cite{Jinbook}, which implies that
$E_{\alpha_0, \alpha_0}(-\lambda  t^{\alpha_0})) \ge 0$ for $t \in [0, T]$, and Lemma \ref{Lem2} and  \eqref{g:est}  to obtain for  $0 < \varepsilon \ll 1$
\begin{equation}\begin{array}{ll}\label{ode:I3}
 \ds \big|\lambda  (t^{\alpha_0-1} \ds E_{\alpha_0, \alpha_0}(-\lambda  t^{\alpha_0}))   *( g^{\prime} * w^{\prime})\big|
 \le Q \lambda I_t^{1-\varepsilon} \big(t^{\alpha_0-1} E_{\alpha_0, \alpha_0}(-\lambda  t^{\alpha_0})\big) *|w^{\prime}|\\[0.1in]
\ds  \le Q \lambda \big(t^{\alpha_0-\varepsilon} E_{\alpha_0, \alpha_0+1-\varepsilon}(-\lambda  t^{\alpha_0}) \big)*|w^{\prime}|\le Q t^{-\varepsilon} *|w^{\prime} |,
  \end{array}
\end{equation}
where we have used the asymptotic property of the Mittag-Leffler function from Lemma \ref{Lem1} such that
$$
\big|\lambda t^{\alpha_0} E_{\alpha_0, \alpha_0+1-\varepsilon}(-\lambda  t^{\alpha_0})\big|\leq Q
$$
holds for some positive constant $Q$ independent from $t$ and $\lambda$.
We again invoke $E_{\alpha_0, \alpha_0}(-\lambda  t^{\alpha_0})) \ge 0$ for $t \in [0, T]$ to bound the last term on the right-hand side of \eqref{diff:sol} and combine the resulting estimate with \eqref{ode:I3}   to bound
\begin{equation}\begin{array}{l}\label{diff:est}
 \hspace{-0.1in} \ds  |v^{\prime}| \le Q\big(t^{-\varepsilon}* |w^{\prime}|+ |f | + \lambda(t^{\alpha_0-1} E_{\alpha_0, \alpha_0}(-\lambda  t^{\alpha_0})) * |f|\big).
  \end{array}
\end{equation}
We combine the estimate
$\int_{0}^{T} e^{-\sigma t} t^{-\mu} d t \le \sigma^{\mu-1} \Gamma(1-\mu)$ for $\mu < 1$
 to bound
\begin{equation}\begin{array}{ll}\label{ode:e2}
\ds  \hspace{-0.125in}  \bigg\|e^{-\sigma t}  \int_0^t \f{|w^{\prime}(s)|ds }{(t-s)^{\varepsilon}}\bigg\|_{L^2(0, T)}  = \big\|(e^{-\sigma t} t^{-\varepsilon}) * (e^{-\sigma t} w^{\prime})\big\|_{L^2(0, T)}
\le \ds Q \sigma^{\varepsilon-1} \|w \|_{\mathcal{X}, \sigma}.
\end{array}
\end{equation}
In addition, we combine  Young's inequality with Lemmas \ref{Lem1}--\ref{Lem2}  to bound the last term on the right-hand side of \eqref{diff:est} as follows
\begin{equation*}\begin{array}{l}
\ds \lambda\big\|(t^{\alpha_0-1} E_{\alpha_0, \alpha_0}(-\lambda  t^{\alpha_0})) * |f| \big\|_{L^2(0,T)}  \le \lambda \big\|t^{\alpha_0-1} E_{\alpha_0, \alpha_0}(-\lambda  t^{\alpha_0})\big \|_{L^1(0, T)}  \|f\|_{L^2(0,T)} \\[0.1in]
\ds \le \lambda T^{\alpha_0}E_{\alpha_0, \alpha_0+1}(-\lambda  T^{\alpha_0})\|f\|_{L^2(0,T)} \le   \|f\|_{L^2(0,T)}.
\end{array}
\end{equation*}
We combine \eqref{diff:est} with the above two  estimates  to obtain
$ \ds  \|v\|_{\mathcal{X}, \sigma} \le Q\big(  \|f\|_{L^2(0, T)}  + \sigma^{\varepsilon-1}  \|w \|_{\mathcal{X}, \sigma} \big)$,
and thus the mapping $\mathcal M : \mathcal{X} \rightarrow \mathcal{X}$ is well-defined.
To prove its contractivity, let $w_1, w_2 \in \mathcal{X}$ with $v_1:=\mathcal{M} w_1$ and $v_2:=\mathcal{M} w_2$. Then $e_v:=v_1-v_2$ and $e_w:=$ $w_1-w_2$ satisfy the homogeneous analogue of the original equation. A direct application of  the above estimate implies  $\left\|e_v\right\|_{\mathcal{X}, \sigma} \leq Q  \sigma^{\varepsilon-1}  \left\|e_w\right\|_{\mathcal{X}, \sigma}$. A sufficiently large $\sigma$ ensures that the mapping $\mathcal{M}$ is a contraction. Thus, the \eqref{ode} admits a unique solution $v \in \mathcal{X}$ with the estimate
$
\|v\|_{\mathcal{X}, \sigma} \leq Q \|f\|_{L^2(0, T)}$,
 which proves the first estimate in  \eqref{odestab}.

In addition,  we directly differentiate \eqref{diff:sol} with $w$ replaced by $v$ and invoke  $v^\prime(0) = f(0)$ from \eqref{diff:sol} to obtain
\begin{equation}\begin{array}{l}\label{Cor:e1}
 \hspace{-0.15in} \ds  v^{\prime \prime}= -\lambda  (t^{\alpha_0-1} E_{\alpha_0, \alpha_0}(-\lambda  t^{\alpha_0})) * g^{\prime}f(0) -\lambda  (t^{\alpha_0-1} E_{\alpha_0, \alpha_0}(-\lambda  t^{\alpha_0})) *( g^{\prime} * v^{\prime\prime}) \\
  \ds  + f^\prime -\lambda (t^{\alpha_0-1} E_{\alpha_0, \alpha_0}(-\lambda  t^{\alpha_0}))  f (0)-\lambda (t^{\alpha_0-1} E_{\alpha_0, \alpha_0}(-\lambda  t^{\alpha_0})) * f^\prime = : \sum_{i=1}^5 \hat I_i.\vspace{-0.05in}
  \end{array}
\end{equation}
We combine this with \eqref{g:est}, the estimate below \eqref{ode:e2},  Young's inequality, and the Sobolev embedding $H^1(0, T)  \hookrightarrow C[0,  T]$ to bound $\hat I_1$ and $\hat I_5$ for $0 < \varepsilon \ll 1$ as follows
\begin{equation}\begin{array}{l}\label{Cor:e2}
 \hspace{-0.1in}\ds\|\hat I_1\|_{L^2(0, T)} +\|\hat I_5\|_{L^2(0, T)}
  \\[0.1in]
  \hspace{-0.1in} \ds \! \le \! Q \lambda \big\|t^{\alpha_0-1} E_{\alpha_0, \alpha_0}(-\lambda  t^{\alpha_0})\big \|_{L^1(0, T)} \big( \|t^{-\varepsilon}\|_{L^2(0,T)} |f(0)| \!+\! \|f\|_{H^1}\big)
\! \le \!  Q\|f\|_{H^1(0,T)}.
  \end{array}
\end{equation}
 We  follow the estimate of \eqref{ode:I3} to bound the second term on the right-hand side of \eqref{Cor:e1} as $|\hat I_2 | \le Q  t^{-\varepsilon} *|v^{\prime \prime} |$.
 We then split $t^{\vartheta} = s^{\vartheta} + (t^{\vartheta}-s^{\vartheta})$ and combine the estimate
$|t^{\vartheta}-s^{\vartheta}| \le Q (t-s)^{\vartheta}$
 to obtain
\begin{equation}\begin{array}{ll}\label{ode:e4}
\ds  \hspace{-0.15in}  \bigg\|e^{-\sigma t} t^{\vartheta}\int_0^t \f{|v^{\prime \prime}(s)|ds }{(t-s)^{\varepsilon}}\bigg\|_{L^2(0, T)} &\hspace{-0.1in}\ds \!\le\! \bigg\|\int_0^t \f{e^{-\sigma (t-s)} }{(t-s)^{\varepsilon}}  e^{-\sigma s} s^{\vartheta}  |v^{\prime \prime}(s)|ds \bigg\|_{L^2(0, T)}\\[0.15in]
&\hspace{-0.15in}\ds \ds + \bigg\|e^{-\sigma t} \int_0^t \f{(t-s)^{\vartheta} |v^{\prime \prime}(s)|ds }{(t-s)^{\varepsilon}}\bigg\|_{L^2(0, T)}  =: J_1  + J_2,
\end{array}
\end{equation}
in which $J_1$ could be similarly bounded  as $ J_1 = \big\|(e^{-\sigma t} t^{-\varepsilon}) * (e^{-\sigma t}t^{\vartheta} v^{\prime \prime})\big\|_{L^2(0, T)}
\le \ds Q \sigma^{\varepsilon-1} \|e^{-\sigma t}t^{\vartheta} v^{\prime \prime} \|_{L^2(0, T)}$.
We apply H\"older's inequality and Young's inequality to  bound $J_2$ in \eqref{ode:e4} in an analogous manner as follows
\begin{equation*} \label{ode:J_2}
\begin{array}{l}
\ds  J_2 =  \ds \bigg\|\int_0^t \f{(t-s)^{\vartheta-\f{\varepsilon}{2}}}{ s^{\vartheta}}\f{e^{-\sigma t} s^{\vartheta} |v^{\prime \prime}(s)|}{(t-s)^{\f{\varepsilon}{2}} }  ds \bigg\|_{L^2(0, T)} \le Q \bigg\|\bigg(\int_0^t  \f{e^{-2\sigma t} s^{2\vartheta} |v^{\prime \prime}(s)|^2}{(t-s)^{\varepsilon}}   ds \bigg)^{\f{1}{2}} \bigg\|_{L^2(0, T)}   \\[0.15in]
\qquad  \ds \le Q \bigg\|\int_0^t  \f{e^{-2\sigma t} s^{2\vartheta} |v^{\prime \prime}(s)|^2}{(t-s)^{\varepsilon}}   ds \bigg\|_{L^1(0, T)}^{\f{1}{2}}  \le Q \bigg\|\int_0^t  \f{e^{-2\sigma (t-s)}}{(t-s)^{\varepsilon}}  e^{-2\sigma s} s^{2\vartheta} |v^{\prime \prime}(s)|^2   ds \bigg\|_{L^1(0, T)}^{\f{1}{2}}\\[0.15in]
 \qquad \le  Q\sigma^{-\f{1-\varepsilon} {2}} \|e^{-\sigma t}t^{\vartheta} v^{\prime \prime} \|_{L^2(0, T)}.
  \end{array}
\end{equation*}
We  invoke Lemma \ref{Lem1}  and $ \f{3}{4} < \gamma <1 $ to bound $\hat I_4$ in \eqref{Cor:e1} as follows
\begin{equation}\begin{array}{l}\label{Cor:e3}
 \hspace{-0.175in}\ds|\hat I_4| =   \big(\lambda^{1-\gamma } t^{\alpha_0(1-\gamma)-1} \big)(\lambda t^{\alpha_0}) ^{\gamma}  E_{\alpha_0, \alpha_0}(-\lambda  t^{\alpha_0})  |f (0)|\le   Q \lambda^{1-\gamma } t^{\alpha_0(1-\gamma)-1} |f (0)|.
  \end{array}
\end{equation}
We then invoke the boundedness of $E_{\alpha_0, \alpha_0}(-\lambda t^{\alpha_0})$ from Lemma \ref{Lem1}, multiply \eqref{Cor:e1} by $e^{-\sigma t} t^{\vartheta}$ on both sides, and combine the estimates \eqref{Cor:e2}--\eqref{Cor:e3} as well as $\vartheta > \f{1}{2} -\alpha_0(1-\gamma)$, that is, $t^{\alpha_0(1-\gamma)-1 + \vartheta} \in L^2(0, T)$, to obtain $$ \big\|e^{-\sigma t}t^{\vartheta}v^{\prime \prime}\big\|_{L^2(0, T)}
  \le   Q\big(\|f\|_{H^1(0,T)} + \lambda^{1-\gamma }  |f(0)|  + \sigma^{-\f{1-\varepsilon} {2}} \|e^{-\sigma t}t^{\vartheta} v^{\prime \prime} \|_{L^2(0, T)} \big).$$
Choose a sufficiently large $\sigma$ to eliminate the last term on the right-hand side of the above inequality to obtain $\big\| t^{\vartheta}v^{\prime \prime}\big\|_{L^2(0, T)}
 \le   Q(\|f\|_{H^1(0,T)} + \lambda^{1-\gamma} |f(0)| )$, which proves the second estimate in \eqref{odestab}.
We thus complete the proof.
\end{proof}

 \section{Analysis of optimal control}\label{Sect:state}

\subsection{Analysis of adjoint  equation}
We prove the well-posedness and solution regularity of the adjoint problem  (\ref{AdjEq}).

\begin{theorem}\label{Cor:Adj}
Suppose $u, u_d \in L^2(L^2)$, the   adjoint problem (\ref{AdjEq}) admits a unique solution and
\begin{equation}\label{Cor:pdestab}\begin{array}{cl}
&\hspace{-0.1in} \ds  \|z\|_{H^1(L^2)}  + \|z\|_{L^2(\check H^2)}  \leq Q  \|u-u_d\|_{L^2(L^2)}.
\end{array}
\end{equation}
In addition, suppose $u, u_d \in L^2(\check H^2)$, we have\vspace{-0.05in}
\begin{equation}\label{Cor:pde:est1}\begin{array}{l}
 \hspace{-0.1in}\ds  \| z\|_{H^1(\check H^2)}  \leq Q \|u-u_d\|_{L^2(\check H^2)}.\vspace{-0.05in}
\end{array}
\end{equation}
Furthermore, suppose $u, u_d \in H^1( L^2)$ with $u(\bm x, T)$, $u_d(\bm x, T) \in \check H^{2-2\gamma}$,  we have
\begin{equation}\label{Cor:pde:est2}\begin{array}{l}
 \hspace{-0.175in}\ds   \big\|(T-t)^{\vartheta} \p_t^2 z \big\|_{L^2(L^2)}   \leq Q (\|u-u_d\|_{H^1(L^2)} + \|u(\cdot, T) \|_{\check H^{2-2\gamma}}+ \|u_d(\cdot, T) \|_{\check H^{2-2\gamma}})\vspace{-0.05in}
\end{array}
\end{equation}
with $\vartheta$ and $\gamma$ defined in \eqref{theta}.
  \end{theorem}
  \begin{proof}
  To analyze the adjoint equation \eqref{AdjEq}, it suffices to consider the equivalent formulation \eqref{AdjEq2} of the  forward-in-time problem  \eqref{AdjEq1} by Theorem \ref{Ref:Adj}.
For $\bar t \in [0, T]$, we expand $\bar z$,  $\bar u$, and $\bar u_d$ in \eqref{AdjEq2} by $\{\phi_i\}_{i=1}^{\infty}$  \cite{SakYam}, that is, $\bar z=\sum_{i=1}^\infty \bar z_i(\bar t)\phi_i(\bm x)$,  $\bar u=\sum_{i=1}^\infty \bar u_i(\bar t)\phi_i(\bm x)$, and $\bar u_d=\sum_{i=1}^\infty \bar u_{d,i}(\bar t)\phi_i(\bm x)$ with $\bar z_i(\bar t) := \big (\bar z(\cdot,\bar t),\phi_i \big)$, $\bar u_i(\bar t) := \big (\bar u(\cdot,\bar t),\phi_i \big)$ and $\bar u_{d,i}(\bar t) := \big (\bar u_d(\cdot,\bar t),\phi_i \big)$. We combine the above expressions with the problem \eqref{AdjEq2} to find that $\bar z$ is the solution to \eqref{AdjEq2} if and only if $\{\bar z_i\}_{i=1}^{\infty}$  satisfy the following equations for $i =1, 2, \ldots$
\begin{equation}\label{pde:e1}
^c\p_{\bar t}^{\alpha_0}  \bar z_i +\lambda_i \bar z_i + \lambda_i g^\prime *  \bar z_i=I_{\bar t}^{1-\alpha_0} (\bar u_i - \bar u_{d,i}),~ \bar t\in (0,T];~~\bar z_i(0)=0,
\end{equation}
which corresponds with the auxiliary differential equation \eqref{ode} with $v = \bar z_i$, $\lambda =\lambda_i$, and $f = \bar u_i - \bar u_{d,i}$, respectively. Then Theorem \ref{thm:ode} gives that \eqref{pde:e1} admits a unique $\bar z_i$ such that the regularity estimates in \eqref{odestab} hold. We combine these  to arrive at a  bound for $\hat z:=\int_0^{\bar t}\big[\sum_{i=1}^\infty \bar z_i'(s)\phi_i(\bm x)\big]ds $\hspace{-0.1in}
 \begin{equation}\begin{array}{l}
\label{pde:e2}
\hspace{-0.2in}\ds \|\p_{\bar t} \hat z \|^2_{L^2(L^2)} \! = \!\sum_{i=1}^\infty \big \|\bar z_i'\|_{L^2(0,T)}^2
\!\leq\!  Q\sum_{i=1}^\infty\|\bar u_i - \bar u_{d,i}\|_{L^2(0, T)}^2 \!=\! Q \|\bar u - \bar u_{d}\|_{L^2(L^2)}^2\\[0.1in]
\ds \qquad \qquad  = Q  \|u-u_d\|_{L^2(L^2)}^2.
\end{array}
\end{equation}
We note that $\hat z$ is the solution to \eqref{AdjEq2} since $\bar z_i(t) := \int_0^{\bar t} \bar z_i'(s)ds $ solves  the ordinary differential equation in (\ref{pde:e1}) for $i \ge 1$.
We thus have
$\|\p_t z\|_{L^2(L^2)} = \|\p_{\bar t} \bar z\|_{L^2(L^2)} \le Q \|u-u_d\|_{L^2(L^2)}.$
To bound $\|\bar z\|_{L^2(\check H^2)}$, we multiply \eqref{AdjEq2} by $e^{-\sigma \bar t}$ for $\sigma >0$ and combine \eqref{g:est}, the estimate below \eqref{diff:est} and \eqref{pde:e2} and  Young's inequality  to obtain
 \begin{equation}\begin{array}{ll}\label{pde:e4}
\ds \hspace{-0.15in}  \ds \|e^{-\sigma \bar t}\Delta \bar z\|_{L^2(L^2)}  \ds = \big\|e^{-\sigma \bar t}\big( {}^c\p_{\bar t}^{\alpha_0}\bar z - g^\prime * \Delta \bar z - I_{\bar t}^{1-\alpha_0} (\bar u -\bar u_d))\big\|_{L^2(L^2)}\\[0.1in]
\ds \le \| \bar t^{-\alpha_0} \|_{L^1(0, T)} \big\|\p_{\bar t} \bar z\|_{L^2(L^2)} \! + \! Q \big(\sigma^{\varepsilon-1} \|\Delta \bar z\|_{L^2(L^2)} \!+ \!\|\bar u -\bar u_d\|_{L^2(L^2)}\big)\\[0.1in]
\ds \le  Q \| u - u_d\|_{L^2(L^2)}+ Q \sigma^{\varepsilon-1} \|\Delta \bar z\|_{L^2(L^2)}.
\end{array}
\end{equation}
Choose a sufficiently large $\sigma$ in \eqref{pde:e4}  to obtain
$\|\Delta \bar z\|_{L^2(L^2)} = \|\bar z\|_{L^2(\check H^2)}  = \| z\|_{L^2(\check H^2)}  \le Q \|u-u_d\|_{L^2(L^2)}$,
which, together with the estimate below \eqref{pde:e2},   proves \eqref{Cor:pdestab}.

 Let $\check z \in H^1(L^2)\cap L^2(\check H^2)$ be another solution to the problem (\ref{AdjEq2}) with the Fourier coefficients $\{\check z_i(\bar t)\}$ for $i=1,2,\cdots$. Then the Fourier coefficients of $\hat z - \check z$ solve the homogeneous analogues of \eqref{AdjEq2}. By the unique solvability of \eqref{pde:e1}, we conclude that $\bar z_i\equiv\check z_i$ for $i \ge 1$, which implies the uniqueness of the solutions of \eqref{AdjEq1}--\eqref{AdjEq2} and the adjoint problem \eqref{AdjEq}. We thus prove the first statement of the theorem.

 We follow the procedures in \eqref{pde:e2} and apply the estimate \eqref{odestab}  to bound
$$ \|  \p_{\bar t} \bar z  \|^2_{L^2(\check H^2)}= \sum_{i=1}^\infty \lambda_i^2  \| \bar z_i^{\prime }\|_{L^2(0,T)}^2\leq  Q\sum_{i=1}^\infty  \lambda_i^2  \|\bar u_i-\bar u_{d,i} \|_{L^2(0, T)}^2 = Q  \|u-u_d\|_{L^2(\check H^2)}^2,$$
which, combined with \eqref{Cor:pdestab} and $\|  \p_t z  \|^2_{L^2(\check H^2)} = \|  \p_{\bar t} \bar z  \|^2_{L^2(\check H^2)}$, proves \eqref{Cor:pde:est1}.

 By \eqref{odestab}, we obtain
 \begin{equation}
 \begin{array}{l}\label{thm2:est:ztt}
\hspace{-0.1in}\ds \big \|(T-t)^{\vartheta}\p_t^2 z \big\|^2_{L^2(L^2)} = \sum_{i=1}^\infty \big \|\bar t^{\vartheta}\bar z_i^{\prime \prime}\big\|_{L^2(0,T)}^2\\[0.1in]
\ds\qquad \qquad \quad \leq  Q\sum_{i=1}^\infty (   \|\bar u_i -\bar u_{d,i}\|_{H^1(0, T)}^2 +  \lambda_i^{2-2\gamma}( |\bar u_i(0)|^2 + |\bar u_{d,i}(0)|^2 )\\[0.15in]
\ds\qquad \qquad \quad   = Q \big(  \|u-u_d\|_{H^1(L^2)}^2+  \|u(\cdot, T)\|_{\check H^{2-2\gamma}}^2 + \|u_d(\cdot, T)\|_{\check H^{2-2\gamma}}^2),
\end{array}
\end{equation}
which proves \eqref{Cor:pde:est2}. We thus complete the proof of the theorem.
  \end{proof}

 \subsection{Analysis of state equation}

 The reformulated version \eqref{Model} of the state equation takes exactly the same form as the problem \eqref{AdjEq2} with $\bar z$, $\bar u - \bar u_d$ replaced by $u$, $q+c$, respectively, and thus the results in Theorem \ref{Cor:Adj}  holds for \eqref{Model} and thus for \eqref{VtFDEs} by Theorem \ref{equi01}, leading to the following estimates.

 \begin{theorem}\label{thm:forward}
 Suppose $q,c\in L^2(L^2)$, then the state equation \eqref{VtFDEs} has a unique solution with the following regularity estimate
\begin{equation}\label{thm:forward:e1}\begin{array}{cl}
&\hspace{-0.1in} \ds   \|   u  \|_{H^1(L^2)} +  \|   u  \|_{L^2(\check H^2)} \leq Q \|q+c\|_{L^2(L^2)}.
\end{array}
\end{equation}
  \end{theorem}

We then prove solution regularity of the  state equation \eqref{VtFDEs}, which differs from the proof of Theorem \ref{Cor:Adj} to accommodate the low regularity of the control variable.

  \begin{theorem}\label{thm:forward2}
Suppose $q, c \in H^1(L^2)$ with $q(\bm x, 0) \in \check H^{2-2\gamma}$, and $u_d \in L^2(\check H^2)$, the following regularity estimates hold
\begin{equation}\label{thm:forward:e2}\begin{array}{ll}
\hspace{-0.2in} \ds \ds  \big\| t^{\vartheta} \p_t^2 u \big\|_{L^2(L^2)}\! +\!\big\| t^{\vartheta} \p_t u \big\|_{L^2(\check H^2)}\!\leq\! Q \big(\|q+c\|_{H^1(L^2)} \!+\! \|u_d\|_{L^2(\check H^{2})} \!+\! \|q(\cdot,0)\|_{\check H^{2-2\gamma}})
\end{array}
\end{equation}
with $\vartheta$ and $\gamma$ defined in \eqref{theta}.
  \end{theorem}
 \begin{proof}
 By \eqref{odestab} and \eqref{thm2:est:ztt}, we similarly obtain
 \begin{equation*}
 \begin{array}{l}\label{thm2:est:utt}
\hspace{-0.1in}\ds \big \|t^{\vartheta}\p_t^2  u \big\|^2_{L^2(L^2)} \le Q \big(  \|q+c\|_{H^1(L^2)}^2+ \|q(\cdot, 0)\|_{\check H^{2-2\gamma}}^2 + \|c(\cdot, 0)\|_{\check H^{2-2\gamma}}^2).
\end{array}
\end{equation*}
We invoke the equivalence between $\check H^{2-2\gamma}$ and $H^{2-2\gamma}$ for $\gamma \in (\f{3}{4}, 1) $ \cite{Jinbook} and \eqref{VarC} to bound
\begin{equation*}\begin{array}{l}
\ds \|c(\cdot,0)\|_{\check H^{2-2\gamma}} \le  \|c(\cdot,0)\|_{ H^{2-2\gamma}} \le \f1{\kappa}\|z(\cdot,0)\|_{ H^{2-2\gamma}}
\ds \leq Q\|z\|_{H^1( \check H^2)}\\[0.1in]
\ds \qquad \qquad \le Q \|u-u_d\|_{L^2(\check H^2)}\le Q \big(\|q+c\|_{L^2(L^2)} + \|u_d\|_{L^2(\check H^2)}\big).
\end{array}\end{equation*}
We combine the above two estimates to get
\begin{equation}\label{thm:forward:e3}\begin{array}{l}
\ds \big\| t^{\vartheta} \p_t^2 u \big\|_{L^2(L^2)}
\ds \le Q(\|q+c\|_{H^1(L^2)} + \|q(\cdot,0)\|_{\check H^{2-2\gamma}} + \|c(\cdot,0)\|_{\check H^{2-2\gamma}}) \\[0.1in]
\ds \qquad \qquad \qquad \qquad \le Q(\|q+c\|_{H^1(L^2)}  + \|u_d\|_{L^2(\check H^2)} +  \|q(\cdot,0)\|_{\check H^{2-2\gamma}}).
\end{array}
\end{equation}

To bound $\p_t \Delta u$, we  differentiate \eqref{VtFDEs} with respect to $t$ and invoke the relation
$ \p_t I_t^{1-\alpha_0} \p_t u =  I_t^{1-\alpha_0} \p_t^2 u + \beta_{1-\alpha_0}(t)\p_t u(\bm x, 0)$ \cite{Jinbook}
 to obtain
\begin{equation}\label{eq2}\begin{array}{l}
\hspace{-0.175in} I_t^{1-\alpha_0}\p_t^2 u -  \p_t \Delta  u - g^\prime * \p_t \Delta  u  \!=\!  I_t^{1-\alpha_0}\p_t (q+c) + \beta_{1-\alpha_0}(t)(q+c-\p_t u)(\bm x, 0).
 \end{array}
\end{equation}
We  incorporate \eqref{ode:e4} to bound the last term on the left-hand side of \eqref{eq2}, and apply \eqref{eq2}, Young's inequality, $\vartheta > \alpha_0 -\f{1}{2}$ and  the Sobolev embedding $H^1(0, T) \hookrightarrow C[0, T]$  to bound $\p_t \Delta u$ in \eqref{eq2}
 \begin{equation*}\begin{array}{ll}\label{thm2:est:e2}
\ds \hspace{-0.15in}  \ds \|e^{-\sigma t} t^{\vartheta} \p_t  u\|_{L^2(\check H^2)}
 & \ds  \le \big\|e^{-\sigma t}t^{\vartheta} (\beta_{1-\alpha_0}\!*\!\p_t^2 u)\big\|_{L^2(L^2)}  \! +\! Q \sigma^{-\f{1-\varepsilon}{2}}  \|e^{-\sigma t} t^{\vartheta} \p_t  u\|_{L^2(\check H^2)}\\[0.1in]
  &\ds    +   \|q+c\|_{H^1(L^2)}  +  Q\| t^{\vartheta -\alpha_0}\|_{L^2(0, T)}\|(q+c-\p_t u)(\cdot, 0)\|_{L^2(\Omega)}\\[0.1in]
    &\ds\le Q \big(\big\|e^{-\sigma t}t^{\vartheta} (\beta_{1-\alpha_0}\!*\!\p_t^2 u)\big\|_{L^2(L^2)}  + \sigma^{-\f{1-\varepsilon}{2}}  \|e^{-\sigma t} t^{\vartheta} \p_t  u\|_{L^2(\check H^2)}\\[0.1in]
    &\ds \qquad \quad +  \|q+c\|_{H^1(L^2)} + \|\p_t u(\cdot, 0)\|_{L^2}\big).
\end{array}
\end{equation*}
We follow the procedures in \eqref{ode:e4} with $\varepsilon$ replaced by $\alpha_0$ to similarly bound the first term on the right-hand side of the above inequality, choose a sufficiently large $\sigma$ to drop the second term, and apply
\begin{equation}\label{u0}\begin{array}{l}
\ds \|\p_t u(\cdot, 0)\|_{L^2} \le \|u\|_{W^{2,1}({L^2})}\leq Q\Big(\|u\|_{H^1(L^2)}+\int_0^T \!\! t^{-\vartheta} \big\|t^{\vartheta}\p_{t}^2 u(\cdot,t)\big\|_{L^2}dt\Big)\\[0.1in]
\ds \qquad \leq Q \big(\|q+c\|_{H^1(L^2)} + \|u_d\|_{L^2(\check H^2)} + \|q(\cdot,0)\|_{ \check H^{2-2\gamma}})
\end{array}
\end{equation}
from \eqref{thm:forward:e3} to obtain\vspace{-0.05in}
\begin{equation*}\begin{array}{ll}\label{thm2:est:e3}
\ds \hspace{-0.15in}  \ds \|t^{\vartheta} \p_t  u\|_{L^2(\check H^2)}
&\ds \ds \le Q \big(\big\| t^{\vartheta} \p_t^2 u \big\|_{L^2(L^2)} + \|q+c\|_{H^1(L^2)} + \|\p_t u(\cdot, 0)\|_{L^2}\big) \\[0.1in]
&\ds \ds \le  Q(\|q+c\|_{H^1(L^2)}  + \|u_d\|_{L^2(\check H^2)}  + \|q(\cdot,0)\|_{ \check H^{2-2\gamma}}),
\end{array}
\end{equation*}
which, combined with \eqref{thm:forward:e3}, proves \eqref{thm:forward:e2}.
 \end{proof}

\subsection{Analysis of optimal control}\label{Sect:Opt}
We prove the well-posedness of the optimal control problem and the regularity of its solutions.

\begin{theorem}\label{thm:OptContrl}
Suppose $q\in H^1(L^2)$, $u_d \in H^1(L^2) \cap L^2(\check H^2)$, and $q(\bm x, 0), u_d(\bm x, T) \in \check H^{2-2\gamma}$. Then the optimal control problem (\ref{ObjFun})--(\ref{VtFDEs}) admits a unique solution $(u, c)$ such that $u \in H^1(L^2) \cap L^2(\check H^2)$. Furthermore, the control variable satisfies $c \in H^1(L^2)\cap L^2(H^2)$, the adjoint  equation \eqref{AdjEq} has a unique solution $z \in H^1(L^2)\cap L^2(\check H^2)$, and the following estimates hold
\begin{equation}\label{OptContrl}\begin{array}{l}
\hspace{-0.15in}\hspace{-0.1in} \|u\|_{H^1(L^2)} +\|u\|_{ L^2(\check H^2)} + \| t^{\vartheta} \p_t^2 u \|_{L^2(L^2)}
+ \|t^{\vartheta}\p_t u \|_{L^2(\check H^2)} + \|u\|_{W^{1,1}(\check H^2)} \\[0.1in]
+\|u\|_{W^{2,1}(L^2)}\le Q (\|q\|_{H^1(L^2)} +\|c\|_{L^2(L^2)} +\|u_d\|_{L^2(\check H^2)} + \|q(\cdot, 0) \|_{\check H^{2-2\gamma}}), \\[0.1in]
\hspace{-0.15in}\hspace{-0.1in} \|z\|_{H^1(L^2)} + \|z\|_{ L^2(\check H^2)} + \big \|(T-t)^{\vartheta}\p_t^2 z \big \|_{L^2(L^2)} +  \| z  \|_{H^1(\check H^2)}\\[0.1in]
\quad   \le Q (\|q\|_{H^1(L^2)} +\|c\|_{L^2(L^2)} +\|u_d\|_{H^1(L^2)} \\[0.1in]
\ds \qquad \qquad \qquad +\|u_d\|_{L^2(\check H^2)} + \|q(\cdot, 0) \|_{\check H^{2-2\gamma}}+ \|u_d(\cdot, T) \|_{\check H^{2-2\gamma}}), \\[0.1in]
\hspace{-0.15in}\hspace{-0.1in}  \|c\|_{H^1(L^2)} \!+\!\|c\|_{ L^2( H^2)}  \le  Q\big(\|u_d\|_{L^2(\check H^2)}+\|q\|_{L^2(L^2)}+\|c\|_{L^2(L^2)} \big).
\end{array}\end{equation}
Here $\vartheta$ and $\gamma$ are defined in \eqref{theta}.
\end{theorem}
\begin{proof}
  The well-posedness proof follows from \cite[Theorem 4.3]{ZheWanSICON} and is thus omitted due to similarity.
By Theorem \ref{thm:forward}, we obtain  $u\in H^1(L^2)\cap L^2(\check H^2)$. Then we apply Theorem \ref{Cor:Adj} to conclude that the adjoint  equation \eqref{AdjEq} has a unique solution $z \in H^1(L^2)\cap L^2(\check H^2)$ with the regularity  estimates \eqref{Cor:pdestab}--\eqref{Cor:pde:est1}. We combine these estimates with \eqref{thm:forward:e1} to obtain
\begin{equation}\label{zt1}\begin{array}{l}
\ds \|z\|_{H^1(L^2)} + \|z\|_{ L^2(\check H^2)} + \|z\|_{H^1(\check H^2)} \le Q( \|q+c\|_{L^2(L^2)} + \|u_d\|_{L^2(\check H^2)}).
    \end{array}
\end{equation}
We combine \eqref{thm:forward:e1} and \eqref{thm:forward:e2}  to bound\vspace{-0.1in}
\begin{equation}\label{uT}\begin{array}{l}
\ds \|u(\cdot, T)\|_{\check H^{2-2\gamma}} \le \|u\|_{W^{1,1}({\check H^2})}\leq Q\Big(\|u\|_{L^1(\check H^2)}+\int_0^T \!\! t^{-\vartheta} \big\|t^{\vartheta}\p_{t}u(\cdot,t)\big\|_{\check H^2}dt\Big)\\[0.1in]
\ds \qquad \leq Q \big(\|q+c\|_{H^1(L^2)} + \|u_d\|_{L^2(\check H^2)} + \|q(\cdot,0)\|_{\check H^{2-2\gamma}}),
\end{array}
\end{equation}
which reformulates  \eqref{Cor:pde:est2}  as follows
\begin{equation}\label{zt2}\begin{array}{l}
 \hspace{-0.175in}\ds   \big\|(T-t)^{\vartheta} \p_t^2 z \big\|_{L^2(L^2)}   \leq Q (\|u-u_d\|_{H^1(L^2)} + \|u(\cdot, T) \|_{\check H^{2-2\gamma}}+ \|u_d(\cdot, T) \|_{\check H^{2-2\gamma}})\\[0.1in]
  \hspace{-0.175in}\ds \ds \!\leq\! Q (\|q\!+\!c\|_{H^1(L^2)} \!+\!\|u_d\|_{H^1(L^2)}\!+ \! \|u_d\|_{L^2(\check H^2)} \!+\! \|q(\cdot, 0) \|_{\check H^{2-2\gamma}} \!+\! \|u_d(\cdot, T) \|_{\check H^{2-2\gamma}}).
\end{array}
\end{equation}
To bound $c$,
we use \eqref{VarC} to obtain $\| c\|_{L^2} \le Q \|z\|_{L^2}$ and $\Delta c = - \Delta z/\kappa$, which, combines \eqref{zt1}, gives
\begin{equation}\label{cxx}
\ds \hspace{-0.1in} \ds\|c\|_{L^2(H^2)} \le Q \|z\|_{L^2(\check
 H^2)}\le Q\big(\|u_d\|_{L^2(\check H^2)}+\|q\|_{L^2(L^2)}+\|c\|_{L^2(L^2)}\big).
\end{equation}
Apply the truncation property of the first order Sobolev space \cite{ZheWanSICON} to \eqref{VarC} to get
$$\begin{array}{rl}
\ds \kappa \p_t c(\bm x,t) & \hspace{-0.1in} \ds = -\p_t z(\bm x,t) + \p_t \max \bigg\{0,\frac{1}{|\Omega|}\int_\Omega z(\bm x,t) d\bm x \bigg \}  = -\p_t z(\bm x,t) \\[0.1in]
\ds & \hspace{-0.1in} \ds + \left\{\begin{array}{ll}
\ds \frac{1}{|\Omega|}\int_\Omega \p_t z(\bm x,t) d\bm x, ~~&\ds \text{a.e. } \bigg \{t \in [0,T]; \int_\Omega z(\bm x,t) d\bm x >0 \bigg \},\\[0.1in]
\ds 0, &\ds \text{a.e. } \bigg \{t\in [0,T]; \int_\Omega z(\bm x,t) d\bm x \leq 0 \bigg \}.
\end{array}
\right. \end{array}$$
We thus arrive at $\|c\|_{H^1(L^2)} \le Q \|z\|_{H^1(L^2)}$, which, together with \eqref{cxx} and \eqref{zt1}, proves the bound for $c$ in \eqref{OptContrl}. We invoke this estimate with \eqref{zt1} and \eqref{zt2} to prove  the estimate for $z$ in  \eqref{OptContrl}.
We invoke \eqref{thm:forward:e2}, \eqref{u0}, and  \eqref{uT} with   the estimate for $c$ in \eqref{OptContrl} to obtain the first estimate in \eqref{OptContrl}. We thus complete the proof.
\end{proof}

\section{Numerical approximation}\label{S:Num}
We derive the numerical approximation to the equivalent formulation (\ref{Model}) of the state equation, which takes the form as \begin{equation}\label{qwl02}
^c\p_t^{\alpha_0} u(\bm x,t) -  \Delta u(\bm x,t)  = (g^\prime * \Delta u) (\bm x,t)  +  F(\bm x,t), \quad F:=  I_t^{1-\alpha_0} (q+c),
\end{equation}
 and accordingly prove its error estimate. Here we assume that $c$ is given a priori, and focus the attention on the error of $u$, which thus reflects the numerical error from discretizing the equation. The equivalent formulation (\ref{AdjEq1}) of the forward-in-time analogue of the adjoint equation \eqref{AdjEq} has exactly the same form as (\ref{qwl02}) (with different forcing terms) and thus we omit the details for its numerical approximation.
\subsection{Time-discrete scheme}
Let $N$ be a positive integer and $\tau=T/N$ be the uniform temporal step size with $t_n=n \tau$ for $0\leq n \leq N$. Let $u^n:=u(\bm x,t_n)$ and ${F}^n : ={F}(\bm x,t_n)$. Discrete $^c\p_t^{\alpha_0} u$ and $(g^\prime * \Delta u)$ at $t = t_n$ by
\begin{equation}\begin{array}{l}\label{Num:e1}
  \ds ^c\p_t^{\alpha_0} u(\bm x,t)  = \sum_{k=1}^n \int_{t_{k-1}}^{t_{k}} \f{\p_s u(\bm x, s) ds}{\Gamma(1-\alpha_0)(t_n-s)^{\alpha_0}} = D_N^{\alpha_0 } u^n + (R_1)^n,\\[0.1in]
  \ds (g^\prime * \Delta u) (\bm x,t_n)   = \sum\limits_{j=1}^{n} \int_{t_{j-1}}^{t_j} g^{\prime}(s)\Delta u (\bm x,t_n-s) ds = \sum\limits_{j=1}^{n} w_j \Delta u^{n-j} + (R_2)^n,
  \end{array}
\end{equation}
where $D_N^{\alpha_0 } u^n$, $(R_1)^n$, $\{w_k\}_{k=1}^n$ and $(R_2)^n$ are defined by
\begin{equation}
\begin{array}{l}\label{Num:e2}
\ds  D_N^{\alpha_0 } u^n: =  \sum_{k=1}^n \int_{t_{k-1}}^{t_k} \f{(u^k-u^{k-1})/\tau ds}{\Gamma(1-\alpha_0)(t_n-s)^{\alpha_0}} = \sum_{k=1}^n  b_{n-k}(u^k-u^{k-1}) \\[0.1in]
 \hspace{0.375in} \ds  = b_{0}u^n+\sum_{k=1}^{n-1}(b_{n-k}-b_{n-k-1})u^k,\quad
  b_{j} \ds := \f{(t_{j+1})^{1-\alpha_0} - (t_j)^{1-\alpha_0}}{\Gamma(2-\alpha_0) \tau}, \vspace{-0.15in}
\end{array}
\end{equation}
\begin{align}
    (R_1)^n & :=  \sum\limits_{k=1}^{n} \int_{t_{k-1}}^{t_k} \frac{1}{\Gamma(1-\alpha_0)(t_n-s)^{\alpha_0}}\left[ \p_s u(\bm x,s) - \frac{u^k-u^{k-1}}{\tau}  \right] ds\label{L1Error} \\
    & =  \sum\limits_{k=1}^{n} \int_{t_{k-1}}^{t_k}\frac{(t_n-s)^{-\alpha_0}}{\tau\Gamma(1-\alpha_0)} \left[ \!\int_{t_{k-1}}^{s}\!\p_t^2 u(\bm x,t)(t\!-\!t_{k-1})dt \!- \!\!\int_{s}^{t_k}\!\p_t^2u(\bm x,t)(t_k\!-\!t) dt   \right]ds, \nonumber\vspace{-0.15in}
\end{align}
and \vspace{-0.1in}
\begin{align}
     w_k  = \int_{t_{k-1}}^{t_k} g'(s) ds,   \quad
   (R_2)^n  = \sum\limits_{j=1}^{n}\int_{t_{j-1}}^{t_j}g^\prime(s)   \int_{s}^{t_j} \p_t\Delta u(\bm x,t_n-\theta)d\theta  ds. \label{IError}\vspace{-0.2in}
\end{align}
We then substitute \eqref{Num:e1} and \eqref{Num:e2} into \eqref{qwl02} to get\vspace{-0.05in}
\begin{equation}\label{qwl010}
  \begin{split}
      D_N^{\alpha_0} u^n - \Delta u^n = \sum\limits_{j=1}^{n} w_j \Delta u^{n-j} + {F}^n + R^n, \quad n = 1,2, \cdots, N
  \end{split}
\end{equation}
with  $R^n=(R_1)^n+(R_2)^n$. Omit the truncation error $R^n$ and replace $u^n$ by its numerical approximation  $U^n$  to get the temporal semi-discrete scheme with $U^0 = 0$
\begin{equation} \label{qwl011}
     \ds D_N^{\alpha_0 } U^n - \Delta U^n = \sum\limits_{j=1}^{n} w_j \Delta U^{n-j} + {F}^n, \quad n = 1,2, \cdots, N.
\end{equation}

\subsection{Analysis of time-discrete scheme}
We  introduce the discrete convolution kernel \cite{LiLiZha} to prove the stability and convergence results of the temporal semi-discrete scheme  \eqref{qwl011}. By \cite{LiLiZha}, we recursively  define the  convolutional coefficients   $\{P_{n-k}\}_{k=1}^{n}$ by
  $P_0 =1$ and $P_{n-k}   = \sum_{j=k+1}^n(b_{j-k-1}-b_{j-k}) P_{n-j}$ for $1 \leq k \leq n-1 $, which have the following properties.
\begin{lemma} \cite[Lemma 2.1]{LiLiZha}\label{liao1}
The discrete coefficients  $\{P_{n-k}\}_{k=1}^{n}$ satisfy
$0<P_{n-k} \leq \Gamma(2-\alpha_0) \tau^{\alpha_0}$ for   $1 \leq k \leq n$. In addition, the following two relations hold: \par
 \text{(i)} $\sum_{j=k}^n P_{n-j}  b_{j-k} =1$  for $1 \leq k \leq n \le N$. \par
 \text{(ii)}
  $\sum_{j=1}^n P_{n-j} \beta_{1+m \alpha_0-\alpha_0} (t_j) \leq \beta_{1+m \alpha_0} (t_n )$ for $1 \le n \le N$ and $m=0,1$.
\end{lemma}


\begin{theorem}\label{thm4.1}
  Let $U^n$  be the solution of the scheme \eqref{qwl011}. The following stability estimate holds
     \begin{align}
\|U^m\| \leq Q\big( \sup\limits_{0\leq t\leq T}\|q(\cdot,t)\|+\sup\limits_{0\leq t\leq T}\|c(\cdot,t)\| \big), \quad  1\leq m \leq N. \nonumber
    \end{align}
\end{theorem}
\begin{proof}
The proof can be performed by the same manner as that of \cite[Theorem 1]{QiuZhe} and is thus omitted.
\end{proof}

Denote $\rho^n = u^n - U^n$ for $n=0$, $1$, $2$, $\cdots$, $N$. We subtract \eqref{qwl011} from \eqref{qwl010} to obtain the following error equations for  $\rho^0 = 0$ and $n = 1, 2, \cdots, N$ \vspace{-0.05in}
\begin{align}
    & D_N^{\alpha_0 } \rho^n - \Delta \rho^n = \sum\limits_{j=1}^{n} w_j \Delta \rho^{n-j} + R^n. \vspace{-0.15in}\label{yq03}
\end{align}

The truncation errors of the scheme (\ref{qwl011}) under graded meshes are analyzed in \cite{LiLiZha,QiuZhe} based on certain solution regularity conditions. In this work, the  truncation errors are re-estimated based on the enhanced solution regularity, which thus relaxes the mesh condition from the graded mesh to the uniform mesh without deteriorating the accuracy.

\begin{lemma}\label{liao2} Assume that $c \in L^2(L^2)$, $q \in H^1(L^2)$, $u_d \in L^2(\check H^2) \cap H^1(L^2)$, and $q(\bm x, 0), u_d(\bm x, T)\in \check H^{2-2\gamma}$  with $\gamma$ defined in \eqref{theta}, we have
\begin{align}
    \sum_{n=1}^m P_{m-n} \|(R_1)^n\| \le Q  \tau, \quad  \sum_{n=1}^m P_{m-n} \|(R_2)^n\| \le Q \ell_{\tau} \tau\label{lem4.2}
\end{align}
 with $1\leq m \leq N$ and $\ell_{\tau} := 1+|\ln\tau|$.
\end{lemma}
\begin{proof} We split the left-hand side term of the first estimate in \eqref{lem4.2} into two parts
$$\sum_{n=1}^m P_{m-n} \|(R_1)^n\|=  P_{m-1} \|(R_1)^1\| + \sum_{n=2}^m P_{m-n} \|(R_1)^n\|,$$
and aim to prove that each term on its right-hand side can be bounded by $Q \tau$.

  For $n=1$, we first use \eqref{L1Error} to get
    \begin{align*}
          (R_1)^1 =  \int_{0}^{\tau} \frac{(\tau-s)^{-\alpha_0}}{\tau\Gamma(1-\alpha_0)}\left[ \int_{0}^{s}\p_t^2 u(\bm x,t)t dt - \int_{s}^{\tau}\p_t^2 u(\bm x,t) (\tau-t) dt   \right]ds.
    \end{align*}
We combine the assumptions of the lemma, Lemma \ref{liao1} and Theorem \ref{thm:OptContrl} to obtain
\begin{align}
  \hspace{-0.1in}  P_{m-1} \|(R_1)^1\| & \leq Q\tau^{\alpha_0} \int_0^{\tau} \frac{ds}{(\tau-s)^{\alpha_0}} \int_0^{\tau} \|\p_t^2u(\cdot, t)\|dt  \leq Q\tau\|\p_t^2u\|_{L^1(L^2)} \leq Q \tau. \label{PP1}
\end{align}

For $n\geq 2$, we  interchange the order of integration to rewrite \eqref{L1Error} as
\begin{align*}
    (R_1)^n = \frac{1}{\Gamma(2-\alpha_0)}  \sum_{k=1}^n \int_{t_{k-1}}^{t_k} \left\{ g(s) - \Big[ \frac{s-t_{k-1}}{\tau}g(t_{k}) + \frac{t_{k}-s}{\tau}g(t_{k-1}) \Big] \right\}  \p_s^2u(\bm x,s)ds,
\end{align*}
where $g(s) = (t_n-s)^{1-\alpha_0}$. By the linear interpolation, we have
\begin{align*}
    g(s) &- \Big[ \frac{s-t_{k-1}}{\tau}g(t_{k}) + \frac{t_{k}-s}{\tau}g(t_{k-1}) \Big] = \frac{1}{2} g''(\zeta_k) (s-t_k)(s-t_{k-1}) \\
    & \leq \frac{\alpha_0(1-\alpha_0)}{2}\tau^2 (t_n-\zeta_k)^{-\alpha_0-1}, \quad \zeta_k\in (t_{k-1},t_k), \quad 1\leq k\leq n-1.
\end{align*}
Thus, we combine the above two formulas and $g(t_n)=0$ to obtain
\begin{align}
     \|(R_1)^n\| &  \leq \frac{1}{\Gamma(2-\alpha_0)}  \int_{t_{n-1}}^{t_n} \left| g(s) - \Big[ 0 + \frac{t_{n}-s}{\tau}g(t_{n-1}) \Big] \right|  \|\p_s^2u(\cdot, s)\|ds \label{nota1} \\
    & + \frac{\alpha_0}{2\Gamma(1-\alpha_0)} \tau^2   \sum_{k=1}^{n-1}  \int_{t_{k-1}}^{t_k} (t_n-\zeta_k)^{-\alpha_0-1} \|\p_s^2u(\cdot, s)\|ds  =:  \|(R_{11})^n\| + \|(R_{12})^n\|.  \nonumber
\end{align}
By Theorem \ref{thm:OptContrl} and Lemma \ref{liao1}, we combine \eqref{nota1} to obtain
    $$\sum\limits_{n=2}^{m} P_{m-n} \|(R_{11})^n\|  \leq Q\tau^{\alpha_0} \int_{t_{1}}^{t_N} (2\tau^{1-\alpha_0})  \|\p_s^2u(\cdot, s)\|ds  \leq Q\tau.$$
We interchange the order of summation, and invoke Lemma \ref{liao1} (ii) and Theorem \ref{thm:OptContrl} to bound the second term on the right-hand side of \eqref{nota1} as follows
\begin{align}
    \sum\limits_{n=2}^{m} P_{m-n} \|(R_{12})^n\| & \leq Q\tau^{2}  \sum\limits_{n=2}^{m} P_{m-n} \sum_{k=1}^{n-1}  (t_n-t_k)^{-\alpha_0-1} \int_{t_{k-1}}^{t_k}   \|\p_s^2u(\cdot, s)\|ds \nonumber \\
    &  \leq Q\tau  \sum\limits_{n=2}^{m} P_{m-n} \sum_{k=1}^{n-1}  (t_n-t_k)^{-\alpha_0} \Big(\int_{t_{k-1}}^{t_k}   \|\p_s^2u(\cdot, s)\|ds \Big) \nonumber \\
    &  = Q\tau \sum_{k=1}^{m-1} \Big(\int_{t_{k-1}}^{t_k}   \|\p_s^2u(\cdot, s)\|ds \Big) \sum\limits_{n=k+1}^{m} P_{m-n} (t_n-t_k)^{-\alpha_0}\nonumber \\
    & \leq Q\tau \sum_{k=1}^{m-1}\int_{t_{k-1}}^{t_k}   \|\p_s^2u(\cdot, s)\|ds\leq Q\tau \|\p_t^2u\|_{L^1(L^2)}\leq Q\tau, \label{nota3}
\end{align}
where we used $$\sum_{n=k+1}^{m} P_{m-n} (t_n-t_k)^{-\alpha_0}=\sum_{j=1}^{m-k} P_{(m-k)-j} (t_j)^{-\alpha_0}\leq Q.$$ We then use \eqref{nota1}--\eqref{nota3} to yield that
$\sum_{n=2}^m P_{m-n} \|(R_1)^n\| \le Q(\alpha_0, T)  \tau$.
By using this and \eqref{PP1}, we prove the first estimate in \eqref{lem4.2}.

To prove the second estimate  in \eqref{lem4.2}, we  apply \eqref{IError} to bound $(R_2)^n$ for $1 \le n \le m \le N$ by
\begin{align}
  \hspace{-0.1in}  \|(R_2)^n\| & \!\le\! \sum\limits_{j=1}^{n}\int_{t_{j-1}}^{t_j} |g^\prime(s)|   \int_{t_{j-1}}^{t_j} \|\p_t\Delta u(\cdot, t_n-t)\| dt  ds
      =: \|(R_{21})^n\| + \|(R_{22})^n\|.\label{qqq3}
\end{align}
where $\|(R_{21})^n\|  : = \int_{0}^{\tau} |g^\prime(s)|   \int_{0}^{\tau}\|\p_t\Delta u(\cdot, t_n-t)\| dt  ds.$
We first apply \eqref{g:est}, Theorem \ref{thm:OptContrl} and Lemma \ref{liao1} to get
\begin{align}
\sum\limits_{n=1}^{m}& P_{m-n}\|(R_{21})^n\| = \sum\limits_{n=1}^{m}P_{m-n} \int_{0}^{\tau} |g^\prime(s)|ds  \Big( \int_{t_{n-1}}^{t_n} \|\p_t\Delta u(\cdot, t)\| dt \Big) \nonumber\\
    & \leq Q\tau^{\alpha_0} \int_{0}^{\tau} s^{-\varepsilon} ds  \Big( \int_{0}^{t_{m}} \|\p_t\Delta u(\cdot, t)\| dt \Big)  \leq Q\tau^{\alpha_0+1-\varepsilon} \|\p_t u\|_{L^1(\check H^2)} \leq Q\tau. \label{R21}
\end{align}
We then use \eqref{g:est}, Theorem \ref{thm:OptContrl} and Lemma \ref{liao1} (ii) to obtain
\begin{align}
    \sum\limits_{n=1}^{m}& P_{m-n}\|(R_{22})^n\| = \sum\limits_{n=2}^{m}P_{m-n} \sum\limits_{j=2}^{n}\int_{t_{j-1}}^{t_j} |g^\prime(s)| ds \Big( \int_{t_{n-j}}^{t_{n-j+1}} \|\p_t\Delta u(\cdot, t)\| dt \Big) \nonumber \\\vspace{-0.1in}
    & \leq Q\tau\ell_{\tau}   \sum\limits_{n=2}^{m}P_{m-n} \int_0^{t_{n-1}} \|\p_t\Delta u(\cdot, t)\| dt \label{R22} \\\vspace{-0.2in}
    & \leq Q\tau\ell_{\tau} \|\p_t u\|_{L^1(\check H^2)}   \sum\limits_{n=2}^{m}P_{m-n} \beta_1(t_n)
   \le Q\tau\ell_{\tau},
     \nonumber\vspace{-0.2in}
\end{align}
where we used the fact that $
\int_{t_{j-1}}^{t_j} |g^\prime(s)| ds\leq Q\tau\ell_{\tau}
$ for $2\leq j\leq n$.
We combine \eqref{qqq3} with the estimates \eqref{R21}--\eqref{R22} to prove the second estimate in \eqref{lem4.2}, and we thus complete the proof.
\end{proof}

By Theorem \ref{thm4.1} and Lemma \ref{liao2}, we prove the following convergence result.

\begin{theorem}\label{thm4.2}
Assume that $c\in L^2(L^2)$, $q \in H^1(L^2)$, $u_d \in L^2(\check H^2) \cap H^1(L^2)$, and $q(\bm x, 0), u_d(\bm x, T)\in \check H^{2-2\gamma}$  with $\gamma$ defined in \eqref{theta}, it holds for \eqref{qwl011}
\begin{equation*}
    \|u-U\|_{\hat L^\infty(L^2)} \leq Q \ell_{\tau}\tau.
\end{equation*}
Here $\ds \|U\|_{\hat L^\infty(L^2)}: = \max_{1 \le n \le N} \|U^n\|$.
\end{theorem}
\begin{proof}
Based on \eqref{yq03},  we follow the proof of \cite[Theorem 1]{QiuZhe} to obtain that for $1 \le m \le N$
 \begin{align}
    \|\rho^m\|^2 + \sum_{n=1}^{m} P_{m-n} \|\nabla \rho^n\|^2  & \leq  2 \sum_{n=1}^{m} P_{m-n} \|R^n\| \|\rho^n\| \nonumber \\
    & \quad + Q \sum\limits_{n=0}^{m-1}  |w_{m-n}| \sum_{j = 1}^{n} P_{n-j}  \|\nabla \rho^{j}\|^2.
        \label{ww012}
\end{align}
We denote
\begin{align*}
   \|\rho^n \|_A :=  \sqrt{\|\rho^n\|^2 + \sum_{j = 1}^{n} P_{n-j}  \|\nabla \rho^{j}\|^2}, \quad 1\leq n \leq N.
\end{align*}
We incorporate this to reformulate \eqref{ww012} as follows
 \begin{align*}
    \|\rho^m \|_A^2  & \leq  2 \sum_{n=1}^{m} P_{m-n} \|R^n\| \|\rho^n \|  + Q \sum\limits_{n=0}^{m-1} | w_{m-n}| \|\rho^n \|_A^2 \\
     & \leq  2 \sum_{n=1}^{m} P_{m-n} \|R^n\| \|\rho^n \|  + Q \sum\limits_{n=0}^{m-1} \mu_{\varepsilon,m-n} \|\rho^n \|_A^2,
\end{align*}
where we used the fact that $|w_n|\leq Q\int_{t_{n-1}}^{t_n}s^{-\varepsilon}ds \leq Q \mu_{\varepsilon,n}$ with $\mu_{\varepsilon,n}:=\frac{t_{n}^{1-\varepsilon}-t_{n-1}^{1-\varepsilon}}{1-\varepsilon}$ for $0 < \varepsilon \ll 1$ and $1 \le n \le N$.
We then invoke \cite[Lemma 6]{Chen} to obtain
 \begin{align}
    \|\rho^{m} \|_A^2  & \leq  2 \sum_{n=1}^{m} P_{m-n} \|R^n\|\|\rho^n \|    + Q(T) \sum\limits_{n=0}^{m-1}  \mu_{\varepsilon,m-n}  \sum_{j=1}^{n} P_{n-j} \|\rho^j \| \|R^j\| \nonumber \\
    &  \leq 2 \sum_{n=1}^{m} P_{m-n} \|R^n\|\|\rho^n \|_A    + Q(T) \sum\limits_{n=0}^{m-1}  \mu_{\varepsilon,m-n}  \sum_{j=1}^{n} P_{n-j} \|\rho^j \|_A \|R^j\|.
    \label{ww015}
\end{align}
Choose $n_0$ such that $\|\rho^{n_0} \|_A = \max\limits_{1\leq n \leq m} \|\rho^n \|_A$. Thus \eqref{ww015} gives

\begin{equation}\label{thm4.2:e3}\begin{array}{ll}
\hspace{-0.1in} \ds \|\rho^{n_0}\|  \ds \le    Q\sum_{n=1}^{n_0} P_{n_0-n} \|R^n\|  +  Q  \sum_{n=0}^{n_0-1}   \mu_{\varepsilon,n_0-n} \sum\limits_{j=1}^{n} P_{n-j}  \|R^j\|.
      \end{array}
\end{equation}
We invoke Lemma \ref{liao2} and \begin{equation}\label{w}
\sum_{n=1}^N  \mu_{\varepsilon,n} = \sum_{n=1}^N  \frac{t_{n}^{1-\varepsilon}-t_{n-1}^{1-\varepsilon}}{1-\varepsilon} = \frac{t_N^{1-\varepsilon}}{1-\varepsilon} \leq Q
\end{equation}
to reformulate \eqref{thm4.2:e3} as
 $\|\rho^{m}\| \le \|\rho^{n_0}\|  \ds \le Q \ell_{\tau} \tau$. We thus complete the proof.
\end{proof}
\subsection{Fully discrete scheme}
We establish and analyze the fully discrete Galerkin scheme for \eqref{qwl02}. Define a quasi-uniform partition of $\Omega$ with the mesh diameter $h$. Let $S_h$ be the space of continuous piecewise linear functions on $\Omega$ with respect to the partition. Denote the Ritz projection $I_h: H^1_0(\Omega)\rightarrow S_h$ by $
     (\nabla (I_h \omega - \omega), \nabla \hat{\chi} ) = 0 $ for $\hat{\chi} \in S_h $
with the approximation property
\begin{align}
   \left\|  \omega - I_h \omega \right\|_{L^2(\Omega)}\leq Q h^2  \| \omega  \|_{H^2(\Omega)}.\vspace{-0.1in} \label{yq06}
\end{align}
    We integrate   \eqref{qwl010} multiplied by $\hat \chi \in H_0^1(\Omega)$ on $\Omega$ to arrive at the weak formulation for $\hat{\chi} \in H_0^1(\Omega)$ and $n = 1, 2, \cdots, N$
    \begin{align}
   \left(D_N^{\alpha_0 } u^n, \hat{\chi} \right) + \left(\nabla u^n, \nabla\hat{\chi} \right)  =  -\sum_{j=1}^{n} w_{j} \left(\nabla u^{n-j}, \nabla\hat{\chi} \right)  +  \left({F}^{n}, \hat{\chi} \right)  +  \left(R^{n}, \hat{\chi} \right). \vspace{-0.1in}\label{FEM:e1}
\end{align}
    We drop the local truncation errors to obtain the fully discrete Galerkin scheme  for \eqref{qwl02}: find  $U^n_h \in S_h$ with $U^0_h : = I_h u_0$ such that for $\hat{\chi} \in S_h$ and $n = 1, 2, \cdots, N$
\begin{align}
   \left(D_N^{\alpha_0 } U^n_h, \hat{\chi} \right) + \left(\nabla U_h^{n}, \nabla\hat{\chi} \right)  =  -\sum_{j=1}^{n} w_{j} \left(\nabla U_h^{n-j}, \nabla\hat{\chi} \right)  +  \left({F}^{n}, \hat{\chi} \right).   \label{lkx01}\vspace{-0.1in}
\end{align}
For the convenience of analysis, we split the error into $u(t_n)- U_h^n =\zeta ^n - \eta ^n$
 with $\zeta^n = I_h u(t_n) - U_h^n \in S_h$ and $\eta ^n = I_h u(t_n) - u(t_n)$ bounded in \eqref{yq06}.
We subtract \eqref{lkx01} from \eqref{FEM:e1} and set $\hat \chi = \zeta^n$ to obtain the following error equation for $n = 1, 2, \cdots, N$
\begin{align}
   \left(D_N^{\alpha_0 }\zeta^n, \zeta^{n} \right) + \left(\nabla \zeta^{n}, \nabla\zeta^{n}\right)  =  -\sum_{j=1}^{n} w_j \left(\nabla \zeta ^{n-j}, \nabla \zeta^{n} \right)  +  \left(R^{n}+D_N^{\alpha_0}\eta^n, \zeta^{n} \right).   \label{lkx03}\vspace{-0.2in}
\end{align}
We then prove the stability result and error estimate for the fully discrete scheme.

\begin{theorem}\label{thm5.4}  Let $U_h^n$ be the solution of the fully discrete scheme \eqref{lkx01}. Assume that $c\in L^2(L^2)$, $q \in H^1(L^2)$, $u_d \in L^2(\check H^2) \cap H^1(L^2)$, and $q(\bm x, 0), u_d(\bm x, T)\in \check H^{2-2\gamma}$  with $\gamma$ defined in \eqref{theta}, then the following stability result holds
    \begin{align}
         \|U_h^N\|  \le Q\big( \sup\limits_{0\leq t\leq T}\|q(\cdot,t)\| +\sup\limits_{0\leq t\leq T}\|c(\cdot,t)\| \big). \label{wl01}
    \end{align}
In addition, the following error estimate holds
   \begin{equation}\label{w102}
      \|U_h-u\|_{\hat L^\infty(L^2)}\leq Q \big(  {\ell_{\tau}} \tau +h^2 \big).
   \end{equation}
\end{theorem}

\begin{proof}
We follow the proof of Theorem \ref{thm4.1} and then choose $\hat{\chi}=U_h^n$ in \eqref{lkx01} to obtain \eqref{wl01}. For  \eqref{lkx03}, we follow the proof of Theorem \ref{thm4.2} to conclude that
   \begin{equation}
    \begin{split}
        \hspace{-0.175in} \ds \|\zeta^{n_0}\|  \ds \! \le Q  {\ell_{\tau}} \tau  \! +\!     Q\sum_{n=1}^{n_0} P_{n_0-n} \left\|D_N^{\alpha_0}\eta^n\right\| \! +\!   Q  \sum_{n=0}^{n_0-1}   {\mu_{\varepsilon,n_0-n} }\sum\limits_{k=1}^{n} P_{n-k}  \left\|D_N^{\alpha_0}\eta^k\right\|\label{thm5.4:e1}
    \end{split}
 \end{equation}
 with $\|\zeta^{n_0}\| = \max_{1 \le n \le N} \|\zeta^{n}\|$. To reformulate right-hand side terms of \eqref{thm5.4:e1}, we aim to prove that
 $ \sum_{n=1}^m P_{m-n} \|D_N^{\alpha_0 }\eta^n\| \le Q(\alpha_0,T)  h^2$ for $ 1 \le m \le N$. Then \eqref{Num:e2} gives
\begin{align}
    D_N^{\alpha_0 }\eta^n = \frac{1}{\tau\Gamma(1-\alpha_0)} \sum\limits_{k=1}^{n}\int_{t_{k-1}}^{t_k} \int_{t_{k-1}}^{t_k} \frac{ (I_h - I) \partial_s u}{(t_n-t)^{\alpha_0}}ds dt. \label{eta}
\end{align}
Similar to the proof of Lemma \ref{liao2}, we split the summation $\sum_{n=1}^m P_{m-n} \|D_N^{\alpha_0 }\eta^n\|$ into two parts. For $n=1$, we invoke \eqref{yq06} to bound \eqref{eta} by
\begin{equation*}
\begin{split}
 \big \| D_N^{\alpha_0 }\eta^1 \big\|  \le \frac{Q h^2}{\tau }  \int_{0}^{\tau}(\tau-t)^{-\alpha_0}dt  \int_{0}^{\tau} \|\p_s u(\cdot, s)\|_{\check H^2} ds \le Q h^2 \tau^{-\alpha_0} \|\p_t u\|_{L^1(\check H^2)},
\end{split}
\end{equation*}
which, together with Theorem \ref{thm:OptContrl} and Lemma \ref{liao1}, gives
\begin{equation} \label{eta1}\begin{split}
   P_{m-1} \big \| D_N^{\alpha_0 }\eta^1 \big\| & \le Q\tau^{\alpha_0}  h^2 \tau^{-\alpha_0} \|\p_t u\|_{L^1(\check H^2)} \le Qh^2.
\end{split}
\end{equation}
In addition, we bound $D_N^{\alpha_0 }\eta^n$ in \eqref{eta} for $2 \le n  \le N$ as follows
\begin{equation} \label{etan}\begin{array}{ll}
\ds \big \| D_N^{\alpha_0 }\eta^n \big\| &
\ds \le \frac{Q h^2}{\tau }  \sum\limits_{k=1}^{n}\int_{t_{k-1}}^{t_k} (t_n-t)^{-\alpha_0} \int_{t_{k-1}}^{t_k}  \|\p_su(\cdot, s)\|_{\check H^2} ds dt\\[0.15in]
&\hspace{-0.1in}
\ds =  \frac{Q h^2}{\tau }   \int_{t_{n-1}}^{t_n} (t_n-t)^{-\alpha_0}  \int_{t_{n-1}}^{t_n}\|\p_su(\cdot, s)\|_{\check H^2} ds dt \\[0.15in]
&\hspace{-0.1in}\ds   + \frac{Q h^2}{\tau }  \sum\limits_{k=1}^{n-1}\int_{t_{k-1}}^{t_k} (t_n-t)^{-\alpha_0} \int_{t_{k-1}}^{t_k}  \|\p_su(\cdot, s)\|_{\check H^2} ds dt = : \hat \Phi_1^n +  \hat \Phi_2^n.
\end{array}
\end{equation}
We employ Theorem \ref{thm:OptContrl} and Lemma \ref{liao1} to bound $\hat \Phi_1^n$ in \eqref{etan} as follows
\begin{equation} \label{etan1}
\begin{split}
 \hspace{-0.1in}  \sum\limits_{n=2}^{m} P_{m-n} \hat \Phi_1^n &\le  \frac{Q h^2}{\tau }  \sum\limits_{n=2}^{m} P_{m-n} \int_{t_{n-1}}^{t_n} (t_n-t)^{-\alpha_0} dt \int_{t_{n-1}}^{t_n} \|\p_su(\cdot, s)\|_{\check H^2} ds \\
   & \leq \frac{Q h^2}{\tau } \tau^{\alpha_0} \tau^{1-\alpha_0} \int_{0}^{t_{m}}\|\p_su(\cdot, s)\|_{\check H^2}ds \leq Qh^2 \|\p_tu\|_{L^1(\check H^2)} \leq Qh^2.
\end{split}
\end{equation}
We swap the order of summation to bound $\hat \Phi_2^n $ in \eqref{etan} by
 \begin{equation*}
 \begin{split}
    \sum\limits_{n=2}^{m} P_{m-n} \hat \Phi_2^n &= \frac{Q h^2}{\tau }
    \sum\limits_{n=2}^{m} P_{m-n}
    \sum\limits_{k=1}^{n-1}\int_{t_{n-k}}^{t_{n-k+1}} t^{-\alpha_0}  dt \int_{t_{k-1}}^{t_k}  \|\p_su(\cdot, s)\|_{\check H^2} ds \\
    & = \frac{Q h^2}{\tau} \sum\limits_{k=1}^{m -1}  \Big(\int_{t_{k-1}}^{t_k}  \|\p_su(\cdot, s)\|_{\check H^2} ds \Big)  \sum\limits_{n=k+1}^{m} P_{m-n}
   \int_{t_{n-k}}^{t_{n-k+1}} t^{-\alpha_0}  dt \\
   & = \frac{Q h^2}{\tau} \sum\limits_{k=1}^{m -1}  \Big(\int_{t_{k-1}}^{t_k}  \|\p_su(\cdot, s)\|_{\check H^2} ds \Big)  \sum\limits_{j=1}^{m -k} P_{(m-k)-j}
   \int_{t_{j}}^{t_{j+1}} t^{-\alpha_0}  dt,
\end{split}
\end{equation*}
and then we use Lemma \ref{liao1} (ii) to get
 \begin{equation} \label{etan2}
 \begin{split}
    \sum\limits_{n=2}^{m} P_{m-n} \hat \Phi_2^n
   & \leq Q h^2 \sum\limits_{k=1}^{m -1}  \Big(\int_{t_{k-1}}^{t_k}  \|\p_su(\cdot, s)\|_{\check H^2} ds \Big) \Big( \sum\limits_{j=1}^{m -k} P_{(m-k)-j} t_j^{-\alpha_0} \Big) \\
   & \leq Q h^2 \sum\limits_{k=1}^{m-1}  \Big(\int_{t_{k-1}}^{t_k} \|\p_su(\cdot, s)\|_{\check H^2} ds \Big) \leq Q h^2 \|\p_t u\|_{L^1(\check H^2)} \leq Q h^2.
\end{split}
\end{equation}
We combine \eqref{eta1}, \eqref{etan1} and \eqref{etan2} to prove $ \sum_{n=1}^m P_{m-n} \|D_N^{\alpha_0 }\eta^n\| \le Q  h^2$ for $ 1 \le m \le N$, which, together with  \eqref{thm5.4:e1} and \eqref{w}, gives
 $\|\zeta^{n}\| \le \|\zeta^{n_0}\|  \ds \le Q \big(  {\ell_{\tau}} \tau +h^2 \big)$.
 We invoke this with \eqref{yq06} to prove \eqref{w102}, which completes the proof.
\end{proof}


\section{Simulation of optimal control}
Based on the numerical investigation in \S \ref{S:Num}, we apply the ``first optimize, then discretize'' approach \cite{Gun,Hinze} to approximate  and simulate the optimal control.
\subsection{Discretization scheme}
For approximation of the state equation (\ref{VtFDEs}), we adopt the scheme (\ref{lkx01}) to compute its equivalent formulation (\ref{qwl02}). The only difference is that the fractional integral in $F$ requires discretization for practical computation. To do this,  we apply the linear interpolation quadrature to approximate the integral terms on the right-hand side of (\ref{qwl02})
 $$(I_t^{1-\alpha_0} \varphi)(t_n) = \sum_{j=1}^n \int_{t_{j-1}}^{t_j} \f{(t_n-s)^{-\alpha_0} \varphi(s)ds}{\Gamma(1-\alpha_0)} = \sum_{j=1}^n\left[\hat a_{n, j} \varphi^j+\hat b_{n, j} \varphi^{j-1}\right]:= \mathcal {I}_n(\varphi),$$ where
$$\hat a_{n, j}=\int_{t_{j-1}}^{t_j} \frac{\left(t_n-s\right)^{-\alpha_0}}{\Gamma\left(1-\alpha_0\right)} \frac{s-t_{j-1}}{\tau} d s, \quad \hat b_{n, j}=\int_{t_{j-1}}^{t_j} \frac{\left(t_n-s\right)^{-\alpha_0}}{\Gamma(1-\alpha_0)} \frac{t_j-s}{\tau} d s
$$ for $n\ge 1$.
 For approximation of the adjoint equation (\ref{AdjEq}), we compute the equivalent formulation (\ref{AdjEq1})  of its  forward-in-time analogue following the scheme (\ref{lkx01}).
We summarize these results to get the fully-discrete finite element scheme for the first-order optimality condition  \eqref{VtFDEs}, \eqref{AdjEq}, \eqref{VarC} of the optimal control: find $ U=\{U^n_h\}_{n=1}^N\subset S_h$  with $U^0_h=0$, $Z = \{Z^n_h\}_{n=0}^{N-1} \subset S_h$ with $Z_h^N = 0$ and $ C=\{C^{n}_h\}_{n=0}^{N-1}$ such that for $n=1,2,\ldots,N$ and $\hat \chi\in S_h$\vspace{-0.05in}
\begin{subequations}
\begin{align}\vspace{0.3in}
\hspace{-0.1in}&\hspace{-0.1in} \ds \ds\left(D_N^{\alpha_0 } U^n_h, \hat{\chi} \right) + \left(\nabla U_h^{n}, \nabla\hat{\chi} \right)  \!=\!  -\sum_{j=1}^{n} w_{j} (\nabla U_h^{n-j}, \nabla\hat{\chi})  + (\mathcal I_n(q+C_h),  \hat{\chi}  ), \label{FEM:u}\vspace{-0.05in} \\[0.05in]
\hspace{-0.15in}&\hspace{-0.1in} \ds   \left(D_N^{\alpha_0 } \bar Z^n_h, \hat{\chi} \right) \!+\! \left(\nabla \bar Z^n_h, \nabla\hat{\chi} \right)  \!=\!  \!-\!\sum_{j=1}^{n} w_{j} (\nabla \bar Z_h^{n-j}, \nabla\hat{\chi} ) \! +  (\mathcal I_n (\bar U_h- \bar u_d), \hat \chi ), \vspace{-0.05in}\label{adjFEM}\\[0.05in]
\hspace{-0.2in}&\ds  C^{n-1}(\bm x) = \f1{\kappa}\max\bigg\{0,\frac{1}{|\Omega|}\int_\Omega Z_h^{n-1}(\bm x) d\bm x\bigg\} - \f1{\kappa}Z_h^{n-1}(\bm x). \vspace{-0.05in}\label{dVarInq:e1}
\end{align}
\end{subequations}
Here $\bar U_h^n: = U_h^{N-n}$ and $\bar Z_h^n : =  Z_h^{N-n}$ for $0 \le n \le N$ and $\bar u_d$ is defined in \eqref{AdjEq1}.

 In practical computations, we start from an initial guess of $C$ to solve  \eqref{FEM:u}, \eqref{adjFEM}, and \eqref{dVarInq:e1} in order and then update $C$ for the next round computation until the difference of two contiguous $C$ under the space-time discrete  infinity norm is less than the tolerance $ 10^{-6}$.

\subsection{Numerical experiments}
We investigate the convergence behavior of the schemes \eqref{lkx01} and \eqref{FEM:u}--\eqref{dVarInq:e1} to the state equation \eqref{VtFDEs} and the optimal control \eqref{ObjFun}--\eqref{VtFDEs}, respectively, and then compare the exact and numerical solutions to substantiate the effectiveness of the scheme \eqref{FEM:u}--\eqref{dVarInq:e1} in simulating optimal control.

Let $\Omega = (0, 1)$,  $M = 1/h$, and define the temporal errors of the state variable  and the corresponding temporal convergence rates as
 $$E^\tau_U(N,M)= \max\limits_{1\leq n \leq N} \sqrt{h\sum_{j=1}^{M-1} \left(U^{2n}_j-U^n_j\right)^2}, \quad \text{Rate}^\tau_U = \log_{2}\bigg(\frac{E_\tau(N,M)}{E_\tau(2N,M)}\bigg),$$
and  similarly define those in space by
 $$E^h_U(N,M)=   \max\limits_{1\leq n \leq N}\sqrt{h\sum_{j=1}^{M-1}\left(U_{2j}^{n}-U_j^n\right)^2}, \quad \text{Rate}^h_U = \log_{2}\bigg(\frac{E_h(N,M)}{E_h(N,2M)}\bigg),$$
  in which the order of convergence is calculated by two-mesh method, see \cite[p.~107]{Farrell}.
We similarly define temporal errors $E^\tau_{{Z}}$, and $ E^\tau_C$ as well as their corresponding temporal convergence rates $\text{Rate}^\tau_{{Z}}$ and $\text{Rate}^\tau_{C}$, and the spatial errors  $E^h_{{Z}}$,  $E^h_C $ and the corresponding convergence rates $\text{Rate}^h_{{Z}}$,  $\text{Rate}^h_C$.

\textbf{Example 1: Convergence of scheme \eqref{lkx01} to  state equation.} Let $T=1/2$, $q=1$, $c=0$, and the variable exponent $\alpha(t)=\alpha_0 - \frac{1}{6}t$.
Numerical results under $M=32$ and $N=64$ are presented in Table \ref{tb01}, which shows the first-order accuracy in time and second-order accuracy in space as proved in Theorem \ref{thm5.4}.

\begin{table}
    \center \footnotesize
    \caption{The errors and convergence rates in Example 1.} \label{tb01}\hspace{-0.1in}
    \begin{tabular}{cccccccccccc}
      \hline
  $\alpha_0$  &  $N$ & $E^\tau_U(N,M) $ & $\text{Rate}^\tau_U$ & $M$ & $E^h_U(N,M) $ &
      $\text{Rate}^h_U$   \\
     \hline
         & 128 &  $3.3834 \times 10^{-5}$   &  *       & 8   &  $8.1507\times 10^{-4}$   & *    \\
  $0.4$  & 256 &  $1.8048 \times 10^{-5}$   &  0.91    & 16  &  $2.1260\times 10^{-4}$   & 1.94  \\
         & 512 &  $9.1291 \times 10^{-6}$   &  0.98    & 32  &  $5.4287\times 10^{-5}$   & 1.97  \\
         & 1024&  $4.4025 \times 10^{-6}$   &  1.05    & 64  &  $1.3716\times 10^{-5}$   & 1.98 \\
     \hline
         & 128 &  $9.7381 \times 10^{-5}$   &  *       & 8   &  $1.0108\times 10^{-3}$   & *    \\
  $0.7$  & 256 &  $4.2862 \times 10^{-5}$   &  1.18    & 16  &  $2.6336\times 10^{-4}$   & 1.94  \\
         & 512 &  $1.8348 \times 10^{-5}$   &  1.22    & 32  &  $6.7208\times 10^{-5}$   & 1.97  \\
         & 1024&  $7.6958 \times 10^{-6}$   &  1.25    & 64  &  $1.6975\times 10^{-5}$   & 1.99 \\
     \hline
         & 128 &  $2.4991 \times 10^{-4}$   &  *       & 8   &  $1.2900\times 10^{-3}$   & *    \\
  $0.95$ & 256 &  $1.2262 \times 10^{-4}$   &  1.03    & 16  &  $3.3474\times 10^{-4}$   & 1.95  \\
         & 512 &  $6.0049 \times 10^{-5}$   &  1.03    & 32  &  $8.5225\times 10^{-5}$   & 1.97  \\
         & 1024&  $2.8964 \times 10^{-5}$   &  1.05    & 64  &  $2.1499\times 10^{-5}$   & 1.99 \\
     \hline
    \end{tabular}

\end{table}

\textbf{Example 2: Convergence of scheme \eqref{FEM:u}--\eqref{dVarInq:e1} to optimal control.} Let   $T=1$, $q=1$,  $\alpha(t)=0.1 - \frac{1}{6}t$, $\kappa=7/8$ and $u_d = 1 - 4(x - 1/2)^2$. We choose $M=32$ and $N=16$, and present the numerical results in Table \ref{tb02}, from which we observe that the errors converge as either the temporal or spatial partition grid is refined, which demonstrates the effectiveness and efficiency of the numerical scheme \eqref{FEM:u}--\eqref{dVarInq:e1} for the optimal control problem \eqref{ObjFun}--\eqref{VtFDEs}.


\begin{table}
	\center \footnotesize
	\caption{The errors and  convergence rates in Example 2.} \label{tb02}\hspace{-0.1in}
	\begin{tabular}{ccccccccccccc}
		\hline
		& $N$ & $E^\tau_U(N,M) $ & $\text{Rate}^\tau_U$ & $M$ & $E^h_U(N,M) $ & $\text{Rate}^h_U$ \\
		\hline
		& 4 &  $9.4582 \times 10^{-4}$   &  *     	& 4 &  $4.5114 \times 10^{-3}$   &  *    \\
		& 8 &  $4.9223 \times 10^{-4}$   &  0.94    & 8&  $1.2265\times 10^{-3}$   &  1.88 \\
		& 16&  $2.5887 \times 10^{-4}$   &  0.93     & 16&  $3.1926 \times 10^{-4}$   &  1.94  \\
		& 32&  $1.3668 \times 10^{-4}$   &  0.92    & 32&  $8.1424 \times 10^{-5}$   &  1.97 \\
		\hline
		& $N$ & $E^\tau_{\bar{Z}}(N,M) $ & $Rate^\tau_{Z}$ 	& $M$ & $E^h_{\bar{Z}}(N,M) $ & $Rate^h_{\bar{Z}}$  \\
		\hline
		& 4 &  $ 5.1963 \times 10^{-4}$   &  *      & 4 &  $1.4833 \times 10^{-3}$   &  *    \\
		& 8 &  $ 2.5652 \times 10^{-4}$   &  1.01    & 8&  $3.6500\times 10^{-4}$   &  2.03\\
		& 16 &  $1.2751 \times 10^{-4}$   &  1.01    & 16&  $9.0752 \times 10^{-5}$   &  2.01    \\
		& 32 &  $6.3582 \times 10^{-5}$   &  1.00   & 32&  $2.2656 \times 10^{-5}$   &  2.00  \\
		\hline
		& $N$ & $E^\tau_C(N,M) $ & $Rate^\tau_C$  	& $M$ & $E^h_C(N,M) $ & $Rate^h_C$\\
		\hline
		& 4  &  $ 5.9386 \times 10^{-4}$   &  *     & 4 &  $1.7009 \times 10^{-3}$   &  *   \\
		& 8  &  $ 2.9317 \times 10^{-4}$   &  1.02   & 8&  $4.1714 \times 10^{-4}$   &  2.03\\
		& 16 & $ 1.4573\times 10^{-4}$   &  1.01    & 16&  $1.0372 \times 10^{-4}$   &  2.01    \\
		& 32 & $7.2665 \times 10^{-5}$   &  1.00   & 32&  $2.5893 \times 10^{-5}$   &  2.00   \\
		\hline
	\end{tabular}
	
\end{table}

\begin{figure}
   \setlength{\abovecaptionskip}{0pt}
	\centering
\includegraphics[width=2.5in,height=1.8in]{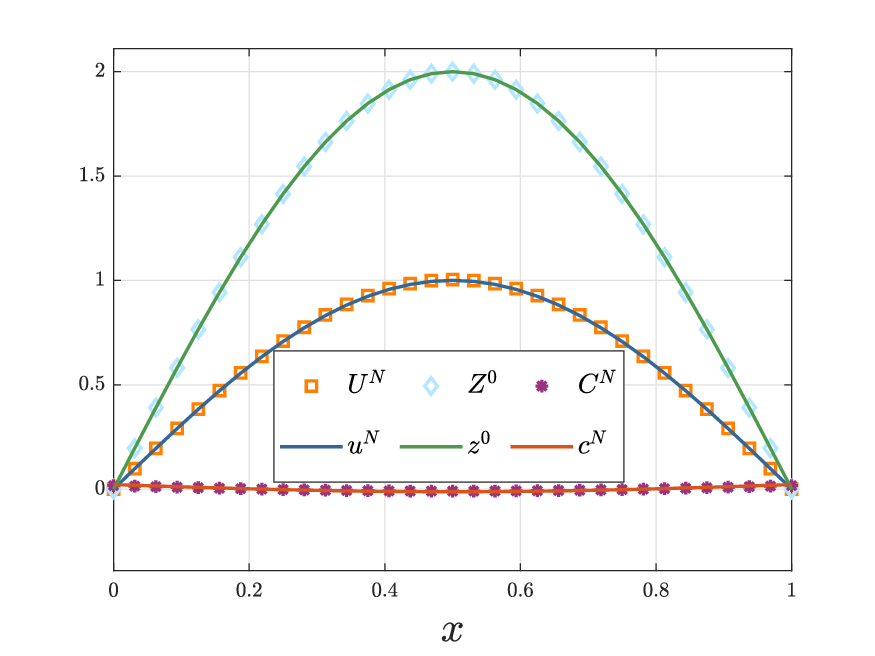}~~~~
\includegraphics[width=2.5in,height=1.8in]{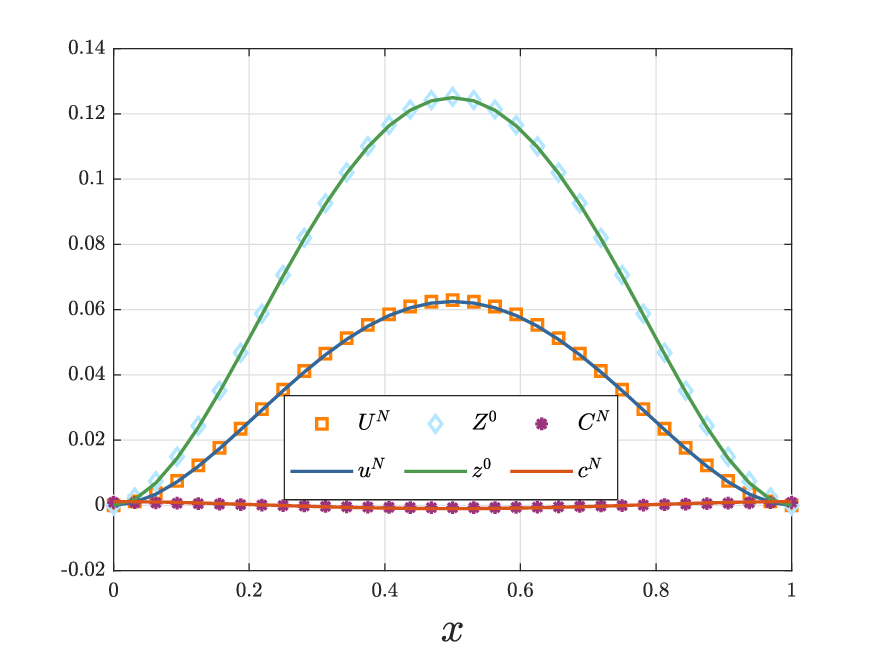}
\caption{Plots of $U^N$, $Z^0$, and $C^N$
in Example 3 for cases (a) (left) and (b) (right).}
\label{figure1}
\end{figure}

\textbf{Example 3: Comparison of exact and numerical solutions to optimal control.} Let   $T=1$, $\kappa=1$ and $\alpha(t) = 0.8 - \f{1}{6}t$. To investigate the accuracy and reliability of the numerical scheme \eqref{FEM:u}--\eqref{dVarInq:e1}, we choose exact solutions of optimal control to be
(a) $u = t^{0.8} \sin(\pi x),\,  z = 2 (1-t)^{0.8} \sin(\pi x)$ or (b)
$u = t^{0.8} x^2(1-x)^2,\,  z = 2 (1-t)^{0.8} x^2(1-x)^2.$
The control variable $c$ could then be explicitly expressed by \eqref{VarC}, and the source term  $q$ as well as $u_d$ could be accordingly determined by  \eqref{Model} and \eqref{AdjEq2}.
In Fig.~\ref{figure1}, we plot the curves of the solutions at the terminal time step by choosing $M=32$ and $N=80$. Numerical results show that the numerical solutions under the
coarse grid still provide an accurate approximation, which again demonstrates the reliability and effectiveness of the numerical scheme to the optimal control problem \eqref{ObjFun}--\eqref{VtFDEs}.

\section{Concluding remarks}
 This work perform a mathematical and numerical investigation to the optimal control of variable-exponent subdiffusion.
A potential extension is to adopt the ``first discretize, then optimize'' approach in numerical approximation of optimal control and the perform numerical analysis. However, the convolution method used in this work introduces an integral operator on the right-hand side of state and adjoint equations, which poses challenges in the derivation of the discrete optimality condition  as well as the error estimate of the optimal control. Another way to reformulate equations is the so-called perturbation method \cite{Zheng}, which performs the kernel splitting and thus preserves the right-hand side terms of the state and adjoint equations, facilitating derivations of the discrete optimality condition.
Nevertheless, such reformulation could not decouple the temporal convolution and the $-\Delta$, which causes additional difficulties in numerical analysis. We will study this interesting topic in the near future.

\section*{Declaration}

\textbf{Conflict of interest} The authors declare no competing interests.
\vspace{0.1in}

\noindent \textbf{Funding} This work was partially supported by the Postdoctoral Fellowship Program of CPSF (No. GZC20240938),  the China Postdoctoral Science Foundation (No. 2024M762459), the Natural Science Foundation of Hubei Province (No. 2025AFB109) and the Shandong Postdoctoral Science Foundation (No. SDZZ-ZR-202501339).

\vspace{0.1in}

\noindent
\textbf{Acknowledgments} Not applicable.
\vspace{0.1in}

\noindent
\textbf{Author contribution} Yiqun Li conceptualized the project, developed the analysis, and wrote the manuscript.
Mengmeng Liu prepared the figures and tables, and wrote the manuscript.
Wenlin Qiu validated the results and wrote the manuscript.
\vspace{0.1in}

\noindent
\textbf{Data Availability} The code supporting the findings of this paper will be made available upon reasonable
request.

\medskip

\end{document}